\documentclass[12pt]{article}
 \usepackage{geometry} 
\usepackage{amsmath,amsthm,amssymb,amsfonts,tikz,mathtools}
\usetikzlibrary{calc,decorations.pathreplacing,arrows}

\newcommand{\degre}{\ensuremath{^\circ}}

\newtheorem{letterthm}{Theorem}
\newtheorem{lettercor}[letterthm]{Corollary}
\newtheorem{letterlem}[letterthm]{Lemma}

\newtheorem{theorem}{Theorem}
\newtheorem{lemma}[theorem]{Lemma}
\newtheorem{corollary}[theorem]{Corollary}
\newtheorem{proposition}[theorem]{Proposition}
\theoremstyle{definition}
\newtheorem{definition}[theorem]{Definition}
\theoremstyle{remark}
\newtheorem{remark}[theorem]{Remark}

\numberwithin{theorem}{section}
 
\begin{document}
\title{Alternating quotients of right-angled Coxeter groups}
\author{Michal Buran}
\date{}
\maketitle

\begin{abstract}
Let $W$ be a right-angled Coxeter group corresponding to a finite non-discrete graph $\mathcal{G}$ with at least $3$ vertices. Our main theorem says that $\mathcal{G}^c$ is connected if and only if for any infinite index convex-cocompact subgroup $H$ of $W$ and any finite subset $\{ \gamma_1, \ldots , \gamma_n \} \subset W \setminus H$ there is a surjection $f$ from $W$ to a finite alternating group such that $f (\gamma_i) \notin f (H)$. A corollary is that a right-angled Artin group splits as a direct product of cyclic groups and groups with many alternating quotients in the above sense.

Similarly, finitely generated subgroups of closed, orientable, hyperbolic surface groups can be separated from finitely many elements in an alternating quotient, answering positively the conjecture of Wilton \cite{wilton2012alternating}.
\end{abstract}

\section{Introduction}
It is often fruitful to study an infinite discrete group via its finite quotients.  For this reason, conditions that guarantee many finite quotients can be useful.

One such notion is residual finiteness. A group $G$ is said to be \emph{residually finite} if for every $g \in G \setminus \{e \}$, there exists a homomorphism $f: G \rightarrow F$, where $F$ is a finite group and $f(g) \neq e$.


We could try to strengthen this notion by requiring that any finite set of non-trivial elements is not killed by some map to a finite group. But these two notions are equivalent as we could simply take product of maps, which don't kill the individual elements.

Another way to modify this is to require that the image of $\gamma$ avoids the image of a specified subgroup $H < G$, which does not contain $\gamma$. If this is true for all finitely generated subgroups $H$, we say that $G$ is \emph{subgroup separable}.

Finitely generated free groups are subgroup separable \cite[Theorem 5.1]{hall1949coset}. The finite quotient $F$ of a free group could be a priori anything. Wilton proved that (for a free group with at least two generators) we can require $f$ to be a surjection onto a finite alternating group, thus giving us some control over the maps which `witness' subgroup separability \cite{wilton2012alternating}.

Scott showed that closed, orientable, hyperbolic surface groups are subgroup separable \cite{scott1978subgroups}. 

Extending and combining methods from both papers, our main theorem shows that even in the case of hyperbolic surface groups, we can require the image to be a finite alternating group.

\begin{definition}
Let $H$ be a subgroup of a finitely generated group $G$, let $\mathcal{C}$ be a class of groups. We say that $H$ is \emph{$\mathcal{C}$-separable} if for any choice of $\{ \gamma_1, \ldots , \gamma_m \} \subset G \setminus H$ there is a surjection $f$ from $G$ to a group in $\mathcal{C}$ such that $f (\gamma_i) \notin f (H)$ for all $i$.
\end{definition}

Note the difference between this terminology and the one above. We talk about subgroups as $\mathcal{C}$-separable in contrast with subgroup separability, which is a property of the entire group.

We will usually take $\mathcal{C}$ to be the class of alternating groups or symmetric groups. We will denote these classes by $\mathcal{A}$ and $\mathcal{S}$, respectively.

In this case, there is a difference between taking a single $\gamma_1$ and multiple group elements as a product of maps surjecting alternating groups is not a map onto an alternating group. In particular, if $G = A_n \times A_m$ then any $\gamma \in G \setminus \{e\}$ does not map to $e$ under at least one of the projections onto factors. However, if we take $\gamma_1, \ldots, \gamma_k$ to be an enumeration of $G \setminus \{e\}$, then the image of any map injective on these elements is isomorphic to $G$ and hence not an alternating group.

The following is our main result.

\begin{letterthm}[Main Theorem] 
Let $\mathcal{G}$ be a non-discrete finite simplicial graph of size at least $3$.
Every infinite index convex-cocompact subgroup of a right-angled Coxeter group $W$ associated to $\mathcal{G}$ is $\mathcal{A}$-separable (and $\mathcal{S}$-separable) if and only if $\mathcal{G}^c$ is connected.
\end{letterthm}

If $\mathcal{G}$ was a discrete graph, there would be difficulties in controlling a permutation parity of the images of generators. It is possible that this can be resolved. A weaker preliminary result, where we can't control parities, still applies in the discrete case. Here ($\mathcal{A} \cup \mathcal{S}$)-separable means that the quotient is required to be either an alternating or a symmetric group.

\begin{letterlem}
Let $\mathcal{G}$ be a finite simplicial graph of size at least $3$.
Every infinite index convex-cocompact subgroup of a right-angled Coxeter group $W$ associated to $\mathcal{G}$ is ($\mathcal{A} \cup \mathcal{S}$)-separable if and only if $\mathcal{G}^c$ is connected.
\end{letterlem}

We require infinite index as otherwise the finite quotient by the normal subgroup contained in $H$ could potentially have no alternating quotients.

Convex-cocompactness is required as not all finitely generated subgroups of RACG are $\mathcal{C}$-separable, where $\mathcal{C}$ is the set of finite groups \cite[Example 10.3]{haglund2008special}.

\begin{lettercor}
Every finitely generated right-angled Artin group is a direct product of cyclic group and groups whose infinite index convex-cocompact subgroups are $\mathcal{A}$-separable.
\end{lettercor}

\begin{lettercor} 
Infinite index finitely generated subgroups of closed, orientable, hyperbolic surface groups are $\mathcal{A}$-separable.
\end{lettercor}

To prove theorem A we'll construct a specific finite sheeted cover of the presentation complex of the RACG; this cover will correspond to a finite index subgroup of the right-angled Coxeter group and the action on the cosets of this subgroup will demonstrate the $\mathcal{A}$-separability.

A right-angled Coxeter group acts on Davis-Moussong complex - which is essentially the Cayley complex with $2$-cells uniquely specified by their boundary and with some higher dimensional cells. A convex-cocompact subgroup $K$ acts cocompactly on some convex subcomplex $Y$ of the Davis-Moussong complex. The group $\Gamma_0$ generated by the reflections in the hyperplanes, which bound $Y$, tiles the complex with the translates of $Y$. Together $\Gamma_0$ and $K$ generate a subgroup, which has index $|K \backslash Y|$ in the RACG. This finite-index subgroup depends on the choice of the convex subcomplex. We will iteratively modify the subcomplex until we arrive at one with quotient of prime size with a long protrusion, which does not contain edges with a certain generator as a label (this is where we use the conditions on $\mathcal{G}$). That generator will fix cosets corresponding to the vertices of the protrusion. Prime size and the element fixing many cosets are ingredients of Jordan's theorem, which says that the action on cosets is symmetric or alternating. Finally, we'll tinker a bit to control the parities of the actions and hence whether the image is symmetric or alternating.

\section{Preliminaries}
\subsection{$\mathcal{A}$-separability}
We will establish some properties of $\mathcal{A}$-separability.
\begin{lemma} \label{product}
Let $A$ and $B$ be non-trivial finitely generated groups. Then $\{e \} < A \times B$ is not $\mathcal{A}$-separable.
\end{lemma}
\begin{proof}
There are only finitely many surjections from $A \times B$ onto $A_2, A_3$ and $A_4$. If $A \times B$ is infinite, then there is a non-identity element $g$ in the kernel of all these maps. Consider elements $g, (e,b), (a,e)$, where $a \neq e$, $b \neq e$. Suppose $f:A \times B \rightarrow A_n$ is a surjection, which does not map these elements to $e$.

By the choice of $g$, we have $n>4$. The group $f(A \times e)$ is a normal subgroup of $A_n$, so it is $e$ or $A_n$. Similarly for $e \times B$. However $A_n$ is not commutative, so one of $A \times e, e \times B$ is mapped to $e$.

If both $A$ and $B$ are finite and $\{e \} < A \times B$ is $\mathcal{A}$-separable, enumerate $A \times B$ as $\gamma_1, \ldots \gamma_m$. Applying the $\mathcal{A}$-separability condition with respect to this set, we get an isomorphism $f : A \times B \rightarrow A_n$.  However, $A_n$ is not a direct product, so one of $A, B$ is $A_n$ and the other is trivial.
\end{proof}

This implies that passing to a finite degree extension does not in general preserve $\mathcal{A}$-separability of convex-cocompact subgroups. However passing to a finite-index subgroup does:

\begin{lemma} \label{fi}
Let $G$ be a finitely generated group, let $H$ be a finite-index subgroup of $G$, and let $K$ be an infinite index subgroup of $H$. If $K$ is $\mathcal{A}$-separable in $G$, then it is $\mathcal{A}$-separable in $H$. 
\end{lemma}
We need $K$ to be infinite index in $H$, as otherwise it is possible that $K = N(H)$ in the notation of the proof below. E.g. take $G = A_n$, $H$ a proper subgroup, $K = \{e\}$.

\begin{proof}
Suppose $\gamma_1, \ldots, \gamma_n \in H \setminus K$.

Let $N(H) = \bigcap_{g \in G} H^g$ be a normal subgroup contained in $H$. Then $N(H)$ is still finite index and let $M = [G : N(H)]$ be this index. Since $G$ is finitely generated, there are only finitely many surjections $f : G \twoheadrightarrow A_m$ with $m \leq M$. The intersection of preimages of $f(K)$ over such surjections is a finite intersection of finite index subgroups, hence a finite index subgroup. So there exists some $\gamma_0 \in G \setminus K$ such that $f(\gamma_0) \in f(K)$ for all $f : G \twoheadrightarrow A_m$ with $m \leq M$.

As $K$ is $\mathcal{A}$-separable in $G$, there exists a surjection $f : G \twoheadrightarrow A_m$, such that $f(\gamma_i) \notin f(K)$ for all $i \in \{0, \ldots n\}$. By the choice of $\gamma_0$ we have $m > M$. But $[A_m:f(N(H))] \leq M$, so $f(N(H)) = A_m$. In particular, $f(H) = A_m$ and $f|_H$ is the desired surjection.
\end{proof}

\subsection{Cube complexes}
For further details of the definitions from this section, the reader is referred to \cite{haglund2008special}.

\begin{definition}[Cube, face]
\emph{An $n$-dimensional cube $C$} is $I^n$, where $I = [-1,1]$. \emph{A face} of a cube is a subset $F = \{\underline{x}: x_i = (-1)^\epsilon \}$, where $1 \leq i \leq n$, $\epsilon =0,1$.
\end{definition}

\begin{definition}[Cube complex]
Suppose $\mathcal{C}$ is a set of cubes and $\mathcal{F}$ is a set of maps between these cubes, each of which is an inclusion of a face. Suppose that every face of a cube in $\mathcal{C}$ is an image of at most one inclusion of a face $f \in \mathcal{F}$.
Then \emph{the cube complex $X$ associated to $(\mathcal{C},\mathcal{F})$} is $$ X = (\bigsqcup_{C \in \mathcal{C}} C) / \sim$$
where $\sim$ is the smallest equivalence relation containing $x \sim f(x)$ for every $f \in \mathcal{F}$, $x \in Dom(f)$.
\end{definition}

\begin{definition}[Midcube]
\emph{A midcube $M$} of a cube $I^n$ is a set of the form $\{\underline{x}: x_i = 0 \}$ for some $ 1 \leq i \leq n$.
\end{definition}

If $f: C \rightarrow C'$ is an inclusion of a face and $M$ is a midcube of $C$, then $f(M)$ is contained in unique midcube $M'$ of $C'$. Moreover $f|_M :M \rightarrow M'$ is an inclusion of a face.

\begin{definition}[Hyperplane]
Let $X$ be a cube complex associated to $(\mathcal{C},\mathcal{F})$. Let $\mathcal{M}$ be the set of midcubes of cubes of $\mathcal{C}$. Let $\mathcal{F'}$  be the set of restrictions of maps in $\mathcal{F}$ to midcubes.

The pair $(\mathcal{M}, \mathcal{F}')$ satisfies that every face is an image of at most one inclusion of a face, so there is an associated cube complex $X'$. Moreover, inclusions of midcubes descend to a map $\varphi : X' \rightarrow X$. \emph{A hyperplane $H$} is a connected component of $X'$ together with a map $\varphi|_H$.
\end{definition}

Hyperplanes are analogous to codimension-1 submanifolds.

\begin{definition}[Elementary parallelism, wall]
Suppose $X$ is a cube complex.

Define a relation of \emph{elementary parallelism} on oriented edges of $X$ by $\overrightarrow{e_1} \sim \overrightarrow{e_2}$ if they form opposite edges of a square. Extend this to the smallest equivalence relation. \emph{The wall $W(\overrightarrow{e})$} is the equivalence class containing $\overrightarrow{e}$. Similarly, we can define an elementary parallelism on unoriented edges and \emph{an unoriented wall $W(e)$}.
\end{definition}

We denote by $\overleftarrow{e}$ the edge $\overrightarrow{e}$ with the opposite orientation.

There is a bijective correspondence between unoriented walls and hyperplanes, where $W(e)$ corresponds to $H(e)$, a hyperplane which contains the unique midcube of $e$.
We say $H(e)$ is dual to $e$. By abuse of notation, we sometimes identify $H(e)$ with its image.

The following notion was invented by Haglund and Wise and was originally called \emph{A-special} \cite[Definition 3.2]{haglund2008special}.

\begin{definition}[Special cube complex]
A cube complex is \emph{special} if the following holds. 
\begin{enumerate}
\item For all edges $\overrightarrow{e} \notin W(\overleftarrow{e})$. We say the hyperplanes are $2$-sided.
\item Whenever $\overrightarrow{e_2} \in W(\overrightarrow{e_1})$, then $e_1$ and $e_2$ are not consecutive edges in a square. Equivalently, each hyperplane embeds.
\item Whenever $\overrightarrow{e_2} \in W(\overrightarrow{e_1})$, $\overrightarrow{e_2} \neq \overrightarrow{e_1}$, then the initial point of $\overrightarrow{e_2}$ is not the initial point of $\overrightarrow{e_1}$. We say that no hyperplane directly self-osculates.
\item Whenever $\overrightarrow{e_2} \in W(\overrightarrow{e_1})$ and $\overrightarrow{f_2} \in W(\overrightarrow{f_1})$ and $e_1$ and $f_1$ form two consecutive edges of a square, if $\overrightarrow{e_2}$ and $\overrightarrow{f_2}$ start at the same vertex, then $\overleftarrow{e_2}$ and $\overrightarrow{f_2}$ are two consecutive edges in some square, and if $\overleftarrow{e_2}$ and $\overrightarrow{f_2}$ start at the same vertex, then $\overrightarrow{e_2}$ and $\overrightarrow{f_2}$ are two consecutive edges in some square. We say that no two hyperplanes inter-osculate.
\end{enumerate}
\end{definition}

Haglund and Wise have shown that $CAT(0)$ cube complexes are special \cite[Example 3.3.(3)]{haglund2008special}. In this paper, we will only ever use specialness of these complexes.

Every special cube complex is contained in a nonpositively curved complex with the same $2$-skeleton \cite[Lemma 3.13]{haglund2008special}. The hyperplane $H(e)$ separates a $CAT(0)$ cube complex $X$ into two connected components. 

\begin{definition} [Half-space, \cite{haglund2008finite}] Suppose $X$ is a cube complex and $H$ is a hyperplane. Let $X \backslash \backslash H$ be the union of cubes disjoint from $H$. If $X$ is $CAT(0)$, $X \backslash \backslash H$ has two connected components. Call them \emph{half-spaces $H^-$ and $H^+$}
\end{definition}

\begin{definition}
Define \emph{$N(H)$} to be the union of all cubes intersecting $H$. Let \emph{$\partial N(H)$} consist of cubes of $N(H)$ that do not intersect $H$. In the case of a simply connected special cube complex $\partial N(H)$ has two components; call them \emph{$\partial N(H)^ +$} and \emph{$\partial N(H)^ -$}. 
\end{definition}

\begin{definition} [Convex subcomplex]
A subcomplex $Y$ of a cube complex $X$ is \emph{(combinatorially geodesically) convex} if any geodesic in $X^{(1)}$ with endpoints in $Y$ is contained in $Y$.
\end{definition}

The components of the boundary of a hyperplane $\partial N(H)^ +$, $\partial N(H)^ -$ and half-spaces  are combinatorially geodesically convex \cite[Lemma 2.10]{haglund2008finite}. Any intersection of half-spaces is convex \cite[Corollary 2.16]{haglund2008finite} and a convex subcomplex of a $CAT(0)$  cube complex coincides with the intersection of all half-spaces containing it \cite[Proposition 2.17]{haglund2008finite}.

%

\begin{definition} [Bounding hyperplane]
A hyperplane \emph{bounds} a convex cubical subcomplex $Y \subset X$ if it is dual to an edge with endpoints $v \in Y$ and $v' \notin Y$.
\end{definition}
\subsection{Right-angled Coxeter and Artin groups}

\begin{definition} [Right-angled Coxeter group]
Given a graph $\mathcal{G}$ with vertex set $I$, let $S = \{ s_i : i \in I \}$. \emph{The right-angled Coxeter group} associated to $\mathcal{G}$ is the group $C(\mathcal{G})$ given by the presentation $\langle S \mid s_i^2 = 1 \text{ for } i \in I, [s_i, s_j]=1  \text{ for } (i,j) \in E(\mathcal{G}) \rangle$. 
\end{definition}

The right-angled Coxeter group $C(\mathcal{G})$ acts on 
\emph{the Davis--Moussong Complex $DM(\mathcal{G})$} \cite{haglund2008special}. Throughout the paper if we talk about the action of $C(\mathcal{G})$ on a cube complex, we mean this action.
The Davis--Moussong complex is similar to Cayley complex, but it does not contain `duplicate squares' and it contains higher dimensional cubes.

Fix a vertex $v_0 \in DM(\mathcal{G})$. There is a bijection between the vertices of $DM(\mathcal{G})$ and the elements of $C(\mathcal{G})$ given by $gv_0 \longleftrightarrow g$.
Vertices $g v_0$ and $g s v_0$ are connected by an edge $ge_s$ labelled $s$. If the generators $s_{i_1}, s_{i_2}, \ldots s_{i_n}$ pairwise commute, there is an $n$-cube with the vertex set $\{ g (\Pi_{j \in P} s_{i_j}) v : P \subset \{1, \ldots, n \} \}$.

Note that $g s_i g^{-1}$ acts on the left on $DM(\mathcal{G})$ as a reflection in $H(ge_{s_i})$. 
There is also a right action of $C(\mathcal{G})$ on $DM(\mathcal{G})^0$, where $s_i$ sends $g v_0$ to $g s_i v_0$ -- the vertex to which $g$ is connected by an edge labelled $s_i$. This action does not extend to $DM(\mathcal{G})$ unless the Coxeter group is abelian.

More generally, if $\Gamma$ is a subgroup of $C(\mathcal{G})$, the action of $C(\mathcal{G})$ on the right cosets of $\Gamma$ can be realised geometrically as an action of $C(\mathcal{G})$ on $\Gamma \backslash DM(\mathcal{G})^0$. This action is given by $ (\Gamma hv_0 ).g= \Gamma hg v_0$.
If $\Gamma$ acts on $DM(\mathcal{G})$ co-compactly, this gives a finite permutation action. We will use this to construct maps from $C(\mathcal{G})$ to $S_n$.

\begin{definition}[Convex-cocompact subgroup]
If $G$ acts on a cube complex $X$, we say $H < G$ is \emph{convex-cocompact} if there is a non-empty convex subcomplex $Y \subset X$, which is invariant under $H$ and moreover $H$ acts on $Y$ cocompactly. We say, that $H$ acts on $X$ with \emph{core} $Y$.
\end{definition}

If $X$ is hyperbolic, this coincides with the usual notion of convex-cocompactness \cite{haglund2008finite}.

\begin{definition}[Right-angled Artin group]
The right-angled Artin group associated to a simplicial graph $\mathcal{G}$ is $A(\mathcal{G}) = \langle g_v: g \in V(\mathcal{G}) \mid g_u g_v = g_v g_u \text{ for } \{u,v\} \in E(\mathcal{G}) \rangle$.
\end{definition}

The next lemma relates RAAGs and RACGs.
\begin{lemma} \cite{davis2000right} \label{RAAGtoRACG}
Given a graph $\mathcal{G}$, define a graph $\mathcal{H}$ as follows:
\begin{itemize}
\item $V(\mathcal{H})=V(\mathcal{G}) \times \{0,1\}$
\item  $(u,1)$ and $(v,1)$ are connected by an edge if $\{u, v\}$ is an edge of $\mathcal{G}$. The vertices $(u,0)$ and $(v,1)$ are connected by an edge if $u$ and $v$ are distinct. Similarly, $(u,0)$ and $(v,0)$ are connected by an edge if $u$ and $v$ are distinct.
\end{itemize}
The right-angled Artin groups $A(\mathcal{G})$ is a finite-index subgroups of the right-angled Coxeter group $C(\mathcal{H})$ via the inclusion $\iota$ extending $g_u \longrightarrow s_{(u,0)} s_{(u,1)}$.

\end{lemma}

\begin{definition}[Salvetti complex]
A right-angled Artin group $A(\mathcal{G})$ acts on Salvetti complex $X=X(\mathcal{G})$, which consists of the following:
\begin{itemize}
\item $X^0 = A(\mathcal{G})$
\item If generators $g_{u_1}, g_{u_2}, \ldots g_{u_n}$ pairwise commute and $g \in A(\mathcal{G})$, there is a unique $n$-cube with the vertex set $\{ g (\Pi_{j \in P} g_{u_{j}}) : P \subset \{1, \ldots, n \} \}$.
\end{itemize}
The action of the right-angled group on the vertex set is by the left multiplication and it extends uniquely to the entire cube complex.
\end{definition}

For the rest of the paper whenever we talk about the action of a RACG or RAAG on a cube complex, we mean the canonical action on the associated Davis-Moussong Complex or Salvetti complex, respectively.

\subsection{Jordan's Theorem}
\begin{definition}[Primitive subgroup]
A subgroup $G < S_n$ is called \emph{primitive} if it acts transitively on $\{1, \ldots ,n \}$ and it does not preserve any nontrivial partition.
\end{definition}
If $n$ is a prime and $G$ is transitive, then the action is primitive.

Our main tool is the following.

\begin{theorem}[Jordan's Theorem] \cite[From theorems 3.3A and 3.3D]{dixon1996permutation}
For each $k>2$ there exists $N$ such that if $n>N$, $G < S_n$ is a primitive subgroup and there exists $\gamma \in G \setminus \{ e \}$, which moves less than $k$ elements,  then $G = S_n$ or $A_n$.
\end{theorem}

\section{The Main Theorem and its consequences}

Our main theorem relates the combinatorics of $\mathcal{G}$ to the $\mathcal{A}$-separability of $C(\mathcal{G})$.

\begin{theorem}[Main Theorem] \label{main}
Let $\mathcal{G}$ be a non-discrete finite simplicial graph of size at least $3$.
Then all infinite-index convex-cocompact subgroups of the right-angled Coxeter group associated to $\mathcal{G}$ are $\mathcal{A}$-separable and $\mathcal{S}$-separable if and only if $\mathcal{G}^c$ is connected.
\end{theorem}

Recall that here convex-cocompact means that it acts cocompactly on a convex subcomplex of the Davis-Moussong complex. A similar result holds for RAAGs.

\begin{corollary} \label{AsepRAAG}
Let $\mathcal{G}$ be a finite simplicial graph of size at least $2$.
Then all infinite index convex-cocompact subgroups of the right-angled Artin group associated to $\mathcal{G}$ are $\mathcal{A}$-separable if and only if $\mathcal{G}^c$ is connected.
\end{corollary}

Here convex-cocompact means that the subgroup acts cocompactly on a convex subcomplex of the Salvetti complex.There is another action of the Artin group on a cube complex given by embedding the group in right-angled Coxeter group as described in the Lemma \ref{RAAGtoRACG}. We will first show that convex-cocompactness with respect to the Salvetti complex implies convex-cocompactness with respect to the Davis-Moussong complex.

\begin{lemma} \label{qcinRACGtoqcinRAAG}
Suppose $\mathcal{G}$ is a simplicial complex, and $K$ a convex-cocompact subgroup of $A(\mathcal{G})$ with respect to the action on $X(\mathcal{G})$. Let $\mathcal{H}$ be as in Lemma \ref{RAAGtoRACG} and identify $A(\mathcal{G})$ with a subgroup of $C(\mathcal{H})$ in the same lemma. Then $K$ is convex-cocompact in $C(\mathcal{H})$ with respect to the action on $DM(\mathcal{H})$.
\end{lemma}

\begin{proof}
Recall that $N(H)$ is the union of all cubes intersecting a hyperplane $H$. For a hyperplane $H$ in a $CAT(0)$ cube complex $X$, $N(H) \simeq H \times [0,1]$. We can collapse $N(H)$ onto $H$. Formally, say $(x,t) \sim (x, t')$ for all $x \in H$ and $t,t' \in [0,1]$. \emph{Collapse of neighbourhood of $H$} is the quotient map $X \longrightarrow X / \sim$.
We can collapse multiple neighbourhoods simultaneously by quotienting by the smallest equivalence relation, which contains the equivalence relation for each hyperplane.

Let $v_0$ be a specified vertex in the Davis-Moussong complex, which under the bijection between vertices and group elements corresponds to the identity. Let $f: (DM(\mathcal{H}), v_0) \longrightarrow (Y,y)$ be the simultaneous collapse of all hyperplanes labelled by $s_{(v,0)}$ for all $v \in \mathcal{G}$. See Figure \ref{tree}. Here, the base point $y$ is the image of $v_0$.  The equivalence relation commutes with the action of $C(\mathcal{H})$, so there is an induced action of $C(\mathcal{H})$ on $Y$. 

We collapsed all edges with labels from $\mathcal{G} \times \{0\}$ so for all $s_{(v,0)}$ and all $g \in C(\mathcal{H})$, we have $gs_{(v,0)}.y = g.y$.

Let $f': X(\mathcal{G}) \longrightarrow Y$ be defined as follows
\begin{itemize}
\item Vertices: Send $g$ to $g.y$.
\item Edges: Send the edge between $g$ and $gg_{v}$ to the edge between $g.y$ and $gg_{v}.y$. It is indeed an edge as $g.y = gs_{(v,0)}.y$ and $gg_v. y =gs_{(v,0)}s_{(v,1)}.y$
\item Squares: Send the square with vertices $g, gg_{v}, gg_u, gg_u g_v$ to the square with vertices $g.y, gg_{v}.y, gg_u.y, gg_u g_v.y$.
\item Higher dimensions: Extend analogously.
\end{itemize}

The right-angled Artin group $A(\mathcal{G})$ acts on $Y$ by $g.(h.y) = gh.y$. The map $f'$ is $A(\mathcal{G})$-equivariant cube complex isomorphism since $g.f'(h)= g. h. y = gh.y = f'(gh)$.

No two hyperplanes of $C(\mathcal{H})$ labelled $s_{(u,0)}$ and $s_{(v,0)}$ osculate since either the neighbourhoods of the associated hyperplanes do not intersect or $u$ is distinct from $v$, $(u,0)$ is connected to $(v,0)$ and the associated hyperplanes intersect.

 I want to prove that if $K$ acts cocompactly on $Z$ a convex subcomplex of $X(\mathcal{G})$, then it acts cocompactly on $W:= f^{-1} f' (Z) \subset DM(\mathcal{H})$. The collapsing map $f$ sends cubes to cubes (of potentially lower dimension), therefore $W$ is a cube complex. To prove cocompactness, it is enough to show that every vertex $y \in Y$ has finitely many vertices in its preimage under $f$. Suppose $x$ and $x'$ are vertices of $DM(\mathcal{H})$ and that they both map to $y$.
Then there is some sequence $H_1, \ldots H_k$ of hyperplanes with labels from $\mathcal{G} \times 0$ and vertices $x_1, \ldots, x_{k+1}$ such that $x_1 = x$, $x_{k+1} = x'$ and $x_i$ maps to the same element as $x_{i+1}$ under the collapse of $H_i$ for all $i$. But then $N(H_i)$ and $N(H_{i+1})$ intersect and as they do not osculate, $H_i$ and $H_{i+1}$ intersect. Since they do not interosculate, $x_{i-1}, x_i$ and $x_{i+1}$ are successive vertices in some square. But now $x_{i+1} \in N(H_{i-1})$ and by induction $H_i$ intersects $H_j$ whenever $i \neq j$. Therefore $H_1, \ldots, H_k$ have distinct labels and $k \leq |\mathcal{G}|$ and the preimage of $y \in Y$ contains at most $2^{|\mathcal{G}|}$ vertices.


It remains to show that $W$ is convex. Let $e$ be an edge in $DM(\mathcal{H})$ with exactly one endpoint in $W$. The edge $e$ is labelled by some $s_{(v,1)}$ as all edges labelled by $s_{(v,0)}$ either lie entirely in $W$ or have an empty intersection with it.
The collapsing map sends parallel edges to parallel edges (unless it sends them both to a vertex) and any sequence of elementary parallelisms in codomain lifts to the domain, so $f(H(e)) = H(f(e))$. In particular, if $H(e)$ intersects $W$, then $H(f(e))$ intersects $Z$ and by the convexity of $Z$, $f(e)$ lies entirely in $Z$, which contradicts that $e$ does not lie entirely in $W$.

So convex-cocompactness with respect to the action on $X(\mathcal{G})$ implies convex-cocompactness with respect to the action on $DM(\mathcal{H})$.
\end{proof}

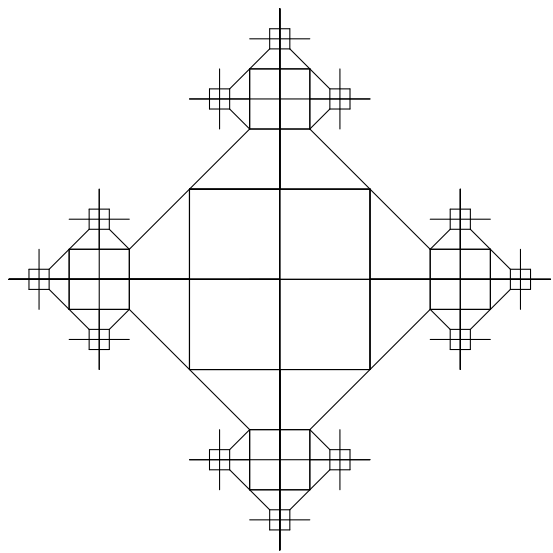
\begin{figure}
\centering
\begin{tikzpicture}[line cap=round,line join=round,>=triangle 45,x=1.0cm,y=1.0cm, scale=1/5]
\clip(-29.91,-23.18) rectangle (27.36,21.97);
\draw (-8,-2)-- (10,-2);
\draw (-2,4)-- (-2,-2);
\draw (-8,-2)-- (-8,16);
\draw (-2,4)-- (-8,4);
\draw (-8,-2)-- (-26,-2);
\draw (-14,4)-- (-14,-2);
\draw (-14,4)-- (-8,4);
\draw (-2,-8)-- (-2,-2);
\draw (-8,-2)-- (-8,-20);
\draw (-2,-8)-- (-8,-8);
\draw (-8,-2)-- (-26,-2);
\draw (-14,-8)-- (-14,-2);
\draw (-14,-8)-- (-8,-8);
\draw (4,-2)-- (10,-2);
\draw (6,0)-- (6,-2);
\draw (4,-2)-- (4,4);
\draw (6,0)-- (4,0);
\draw (4,-2)-- (-2,-2);
\draw (2,0)-- (2,-2);
\draw (2,0)-- (4,0);
\draw (6,-4)-- (6,-2);
\draw (4,-2)-- (4,-8);
\draw (6,-4)-- (4,-4);
\draw (4,-2)-- (-2,-2);
\draw (2,-4)-- (2,-2);
\draw (2,-4)-- (4,-4);
\draw (-2,-2)-- (10,-2);
\draw (4,2)-- (6,2);
\draw (4.67,2.67)-- (4.67,2);
\draw (4,2)-- (4,4);
\draw (4.67,2.67)-- (4,2.67);
\draw (3.33,2.67)-- (3.33,2);
\draw (3.33,2.67)-- (4,2.67);
\draw (4.67,1.33)-- (4.67,2);
\draw (4,2)-- (4,0);
\draw (4.67,1.33)-- (4,1.33);
\draw (3.33,1.33)-- (3.33,2);
\draw (3.33,1.33)-- (4,1.33);
\draw (-10,-2)-- (8,-2);
\draw (2,-2)-- (8,-2);
\draw (2,2)-- (4,2);
\draw (8.67,-1.33)-- (8.67,-2);
\draw (8,-2)-- (8,0);
\draw (8.67,-1.33)-- (8,-1.33);
\draw (7.33,-1.33)-- (7.33,-2);
\draw (7.33,-1.33)-- (8,-1.33);
\draw (-2,4)-- (2,0);
\draw (6,0)-- (7.33,-1.33);
\draw (6,0)-- (4.67,1.33);
\draw (2,0)-- (3.33,1.33);
\draw (-8,-2)-- (-8,16);
\draw (-2,4)-- (-8,4);
\draw (-2,4)-- (-2,-2);
\draw (-8,10)-- (-8,16);
\draw (-6,12)-- (-8,12);
\draw (-8,10)-- (-2,10);
\draw (-6,12)-- (-6,10);
\draw (-8,10)-- (-8,4);
\draw (-6,8)-- (-8,8);
\draw (-6,8)-- (-6,10);
\draw (-10,12)-- (-8,12);
\draw (-10,12)-- (-10,10);
\draw (-8,10)-- (-8,4);
\draw (-10,8)-- (-8,8);
\draw (-10,8)-- (-10,10);
\draw (-8,4)-- (-8,16);
\draw (-4,10)-- (-4,12);
\draw (-3.33,10.67)-- (-4,10.67);
\draw (-4,10)-- (-2,10);
\draw (-3.33,10.67)-- (-3.33,10);
\draw (-3.33,9.33)-- (-4,9.33);
\draw (-3.33,9.33)-- (-3.33,10);
\draw (-4.67,10.67)-- (-4,10.67);
\draw (-4,10)-- (-6,10);
\draw (-4.67,10.67)-- (-4.67,10);
\draw (-4.67,9.33)-- (-4,9.33);
\draw (-4.67,9.33)-- (-4.67,10);
\draw (-4,8)-- (-4,10);
\draw (-7.33,14.67)-- (-8,14.67);
\draw (-8,14)-- (-6,14);
\draw (-7.33,14.67)-- (-7.33,14);
\draw (-7.33,13.33)-- (-8,13.33);
\draw (-7.33,13.33)-- (-7.33,14);
\draw (-2,4)-- (-6,8);
\draw (-6,12)-- (-7.33,13.33);
\draw (-6,12)-- (-4.67,10.67);
\draw (-6,8)-- (-4.67,9.33);
\draw (-8,-2)-- (-26,-2);
\draw (-14,4)-- (-14,-2);
\draw (-14,4)-- (-8,4);
\draw (-20,-2)-- (-26,-2);
\draw (-22,0)-- (-22,-2);
\draw (-20,-2)-- (-20,4);
\draw (-22,0)-- (-20,0);
\draw (-20,-2)-- (-14,-2);
\draw (-18,0)-- (-18,-2);
\draw (-18,0)-- (-20,0);
\draw (-22,-4)-- (-22,-2);
\draw (-22,-4)-- (-20,-4);
\draw (-20,-2)-- (-14,-2);
\draw (-18,-4)-- (-18,-2);
\draw (-18,-4)-- (-20,-4);
\draw (-14,-2)-- (-26,-2);
\draw (-20,2)-- (-22,2);
\draw (-20.67,2.67)-- (-20.67,2);
\draw (-20,2)-- (-20,4);
\draw (-20.67,2.67)-- (-20,2.67);
\draw (-19.33,2.67)-- (-19.33,2);
\draw (-19.33,2.67)-- (-20,2.67);
\draw (-20.67,1.33)-- (-20.67,2);
\draw (-20,2)-- (-20,0);
\draw (-20.67,1.33)-- (-20,1.33);
\draw (-19.33,1.33)-- (-19.33,2);
\draw (-19.33,1.33)-- (-20,1.33);
\draw (-18,2)-- (-20,2);
\draw (-24.67,-1.33)-- (-24.67,-2);
\draw (-24,-2)-- (-24,0);
\draw (-24.67,-1.33)-- (-24,-1.33);
\draw (-23.33,-1.33)-- (-23.33,-2);
\draw (-23.33,-1.33)-- (-24,-1.33);
\draw (-14,4)-- (-18,0);
\draw (-22,0)-- (-23.33,-1.33);
\draw (-22,0)-- (-20.67,1.33);
\draw (-18,0)-- (-19.33,1.33);
\draw (-8,-2)-- (-8,16);
\draw (-14,4)-- (-8,4);
\draw (-14,4)-- (-14,-2);
\draw (-8,10)-- (-8,16);
\draw (-10,12)-- (-8,12);
\draw (-8,10)-- (-14,10);
\draw (-10,12)-- (-10,10);
\draw (-8,10)-- (-8,4);
\draw (-10,8)-- (-8,8);
\draw (-10,8)-- (-10,10);
\draw (-6,12)-- (-8,12);
\draw (-6,12)-- (-6,10);
\draw (-8,10)-- (-8,4);
\draw (-6,8)-- (-8,8);
\draw (-6,8)-- (-6,10);
\draw (-8,4)-- (-8,16);
\draw (-12,10)-- (-12,12);
\draw (-12.67,10.67)-- (-12,10.67);
\draw (-12,10)-- (-14,10);
\draw (-12.67,10.67)-- (-12.67,10);
\draw (-12.67,9.33)-- (-12,9.33);
\draw (-12.67,9.33)-- (-12.67,10);
\draw (-11.33,10.67)-- (-12,10.67);
\draw (-12,10)-- (-10,10);
\draw (-11.33,10.67)-- (-11.33,10);
\draw (-11.33,9.33)-- (-12,9.33);
\draw (-11.33,9.33)-- (-11.33,10);
\draw (-12,8)-- (-12,10);
\draw (-8.67,14.67)-- (-8,14.67);
\draw (-8,14)-- (-10,14);
\draw (-8.67,14.67)-- (-8.67,14);
\draw (-8.67,13.33)-- (-8,13.33);
\draw (-8.67,13.33)-- (-8.67,14);
\draw (-14,4)-- (-10,8);
\draw (-10,12)-- (-8.67,13.33);
\draw (-10,12)-- (-11.33,10.67);
\draw (-10,8)-- (-11.33,9.33);
\draw (-2,-8)-- (-2,-2);
\draw (-8,-2)-- (-8,-20);
\draw (-2,-8)-- (-8,-8);
\draw (-8,-2)-- (-26,-2);
\draw (-14,-8)-- (-14,-2);
\draw (-14,-8)-- (-8,-8);
\draw (-2,4)-- (-2,-2);
\draw (-2,4)-- (-8,4);
\draw (-8,-2)-- (-26,-2);
\draw (-14,4)-- (-14,-2);
\draw (-14,4)-- (-8,4);
\draw (4,-2)-- (10,-2);
\draw (6,-4)-- (6,-2);
\draw (4,-2)-- (4,-8);
\draw (6,-4)-- (4,-4);
\draw (4,-2)-- (-2,-2);
\draw (2,-4)-- (2,-2);
\draw (2,-4)-- (4,-4);
\draw (6,0)-- (6,-2);
\draw (4,-2)-- (4,4);
\draw (6,0)-- (4,0);
\draw (4,-2)-- (-2,-2);
\draw (2,0)-- (2,-2);
\draw (2,0)-- (4,0);
\draw (-2,-2)-- (10,-2);
\draw (4,-6)-- (6,-6);
\draw (4.67,-6.67)-- (4.67,-6);
\draw (4,-6)-- (4,-8);
\draw (4.67,-6.67)-- (4,-6.67);
\draw (3.33,-6.67)-- (3.33,-6);
\draw (3.33,-6.67)-- (4,-6.67);
\draw (4.67,-5.33)-- (4.67,-6);
\draw (4,-6)-- (4,-4);
\draw (4.67,-5.33)-- (4,-5.33);
\draw (3.33,-5.33)-- (3.33,-6);
\draw (3.33,-5.33)-- (4,-5.33);
\draw (2,-6)-- (4,-6);
\draw (8.67,-2.67)-- (8.67,-2);
\draw (8,-2)-- (8,-4);
\draw (8.67,-2.67)-- (8,-2.67);
\draw (7.33,-2.67)-- (7.33,-2);
\draw (7.33,-2.67)-- (8,-2.67);
\draw (-2,-8)-- (2,-4);
\draw (6,-4)-- (7.33,-2.67);
\draw (6,-4)-- (4.67,-5.33);
\draw (2,-4)-- (3.33,-5.33);
\draw (-8,-2)-- (-8,-20);
\draw (-2,-8)-- (-8,-8);
\draw (-2,-8)-- (-2,-2);
\draw (-8,-14)-- (-8,-20);
\draw (-6,-16)-- (-8,-16);
\draw (-8,-14)-- (-2,-14);
\draw (-6,-16)-- (-6,-14);
\draw (-8,-14)-- (-8,-8);
\draw (-6,-12)-- (-8,-12);
\draw (-6,-12)-- (-6,-14);
\draw (-10,-16)-- (-8,-16);
\draw (-10,-16)-- (-10,-14);
\draw (-8,-14)-- (-8,-8);
\draw (-10,-12)-- (-8,-12);
\draw (-10,-12)-- (-10,-14);
\draw (-8,-8)-- (-8,-20);
\draw (-4,-14)-- (-4,-16);
\draw (-3.33,-14.67)-- (-4,-14.67);
\draw (-4,-14)-- (-2,-14);
\draw (-3.33,-14.67)-- (-3.33,-14);
\draw (-3.33,-13.33)-- (-4,-13.33);
\draw (-3.33,-13.33)-- (-3.33,-14);
\draw (-4.67,-14.67)-- (-4,-14.67);
\draw (-4,-14)-- (-6,-14);
\draw (-4.67,-14.67)-- (-4.67,-14);
\draw (-4.67,-13.33)-- (-4,-13.33);
\draw (-4.67,-13.33)-- (-4.67,-14);
\draw (-4,-12)-- (-4,-14);
\draw (-7.33,-18.67)-- (-8,-18.67);
\draw (-8,-18)-- (-6,-18);
\draw (-7.33,-18.67)-- (-7.33,-18);
\draw (-7.33,-17.33)-- (-8,-17.33);
\draw (-7.33,-17.33)-- (-7.33,-18);
\draw (-2,-8)-- (-6,-12);
\draw (-6,-16)-- (-7.33,-17.33);
\draw (-6,-16)-- (-4.67,-14.67);
\draw (-6,-12)-- (-4.67,-13.33);
\draw (-8,-2)-- (-26,-2);
\draw (-14,-8)-- (-14,-2);
\draw (-14,-8)-- (-8,-8);
\draw (-20,-2)-- (-26,-2);
\draw (-22,-4)-- (-22,-2);
\draw (-20,-2)-- (-20,-8);
\draw (-22,-4)-- (-20,-4);
\draw (-20,-2)-- (-14,-2);
\draw (-18,-4)-- (-18,-2);
\draw (-18,-4)-- (-20,-4);
\draw (-22,0)-- (-22,-2);
\draw (-22,0)-- (-20,0);
\draw (-20,-2)-- (-14,-2);
\draw (-18,0)-- (-18,-2);
\draw (-18,0)-- (-20,0);
\draw (-14,-2)-- (-26,-2);
\draw (-20,-6)-- (-22,-6);
\draw (-20.67,-6.67)-- (-20.67,-6);
\draw (-20,-6)-- (-20,-8);
\draw (-20.67,-6.67)-- (-20,-6.67);
\draw (-19.33,-6.67)-- (-19.33,-6);
\draw (-19.33,-6.67)-- (-20,-6.67);
\draw (-20.67,-5.33)-- (-20.67,-6);
\draw (-20,-6)-- (-20,-4);
\draw (-20.67,-5.33)-- (-20,-5.33);
\draw (-19.33,-5.33)-- (-19.33,-6);
\draw (-19.33,-5.33)-- (-20,-5.33);
\draw (-18,-6)-- (-20,-6);
\draw (-24.67,-2.67)-- (-24.67,-2);
\draw (-24,-2)-- (-24,-4);
\draw (-24.67,-2.67)-- (-24,-2.67);
\draw (-23.33,-2.67)-- (-23.33,-2);
\draw (-23.33,-2.67)-- (-24,-2.67);
\draw (-14,-8)-- (-18,-4);
\draw (-22,-4)-- (-23.33,-2.67);
\draw (-22,-4)-- (-20.67,-5.33);
\draw (-18,-4)-- (-19.33,-5.33);
\draw (-8,-2)-- (-8,-20);
\draw (-14,-8)-- (-8,-8);
\draw (-14,-8)-- (-14,-2);
\draw (-8,-14)-- (-8,-20);
\draw (-10,-16)-- (-8,-16);
\draw (-8,-14)-- (-14,-14);
\draw (-10,-16)-- (-10,-14);
\draw (-8,-14)-- (-8,-8);
\draw (-10,-12)-- (-8,-12);
\draw (-10,-12)-- (-10,-14);
\draw (-6,-16)-- (-8,-16);
\draw (-6,-16)-- (-6,-14);
\draw (-8,-14)-- (-8,-8);
\draw (-6,-12)-- (-8,-12);
\draw (-6,-12)-- (-6,-14);
\draw (-8,-8)-- (-8,-20);
\draw (-12,-14)-- (-12,-16);
\draw (-12.67,-14.67)-- (-12,-14.67);
\draw (-12,-14)-- (-14,-14);
\draw (-12.67,-14.67)-- (-12.67,-14);
\draw (-12.67,-13.33)-- (-12,-13.33);
\draw (-12.67,-13.33)-- (-12.67,-14);
\draw (-11.33,-14.67)-- (-12,-14.67);
\draw (-12,-14)-- (-10,-14);
\draw (-11.33,-14.67)-- (-11.33,-14);
\draw (-11.33,-13.33)-- (-12,-13.33);
\draw (-11.33,-13.33)-- (-11.33,-14);
\draw (-12,-12)-- (-12,-14);
\draw (-8.67,-18.67)-- (-8,-18.67);
\draw (-8,-18)-- (-10,-18);
\draw (-8.67,-18.67)-- (-8.67,-18);
\draw (-8.67,-17.33)-- (-8,-17.33);
\draw (-8.67,-17.33)-- (-8.67,-18);
\draw (-14,-8)-- (-10,-12);
\draw (-10,-16)-- (-8.67,-17.33);
\draw (-10,-16)-- (-11.33,-14.67);
\draw (-10,-12)-- (-11.33,-13.33);
\end{tikzpicture} 
\caption{The Salvetti complex for the free group on two generators overlaid with the Davis-Moussong complex for a path of length $3$. The Davis-Moussong complex retracts onto the Salvetti complex by the collapse of hyperplanes.}\label{tree}
\end{figure}

\begin{proof} [Proof of Corollary \ref{AsepRAAG}]
$\Rightarrow$: If $H,K$ are components of $\mathcal{G}^c$, then $A(\mathcal{G}) = A(H^c) \times A(K^c)$ so by Lemma \ref{product} the trivial subgroup $\{e\}$ is not $\mathcal{A}$-separable in $A(\mathcal{G})$.

$\Leftarrow$:  Let $\mathcal{H}$ be as in  Lemma \ref{RAAGtoRACG}.

Suppose $U$ is a  proper component of $\mathcal{H}^c$. The vertices $(v,0)$ and $(v,1)$ are not connected by an edge in $\mathcal{H}$, so $U^0$ is of the form $V \times \{0,1\}$ for some $V \subsetneq \mathcal{G}^0$. But then looking at $V \times \{1\} \subset \mathcal{G} \times \{1\}$ gives that $V^0$ is a vertex set of a proper component of $\mathcal{G}^c$.

So $\mathcal{G}^c$ being connected implies that $\mathcal{H}^c$ is connected.

By Lemma \ref{qcinRACGtoqcinRAAG} $K$ is convex-cocompact in $C(\mathcal{H})$ and hence by \ref{main} it is $\mathcal{A}$-separable in $C(
\mathcal{H})$. By Lemma \ref{fi} $K$ is also $\mathcal{A}$-separable in $A(\mathcal{G})$.
\end{proof}

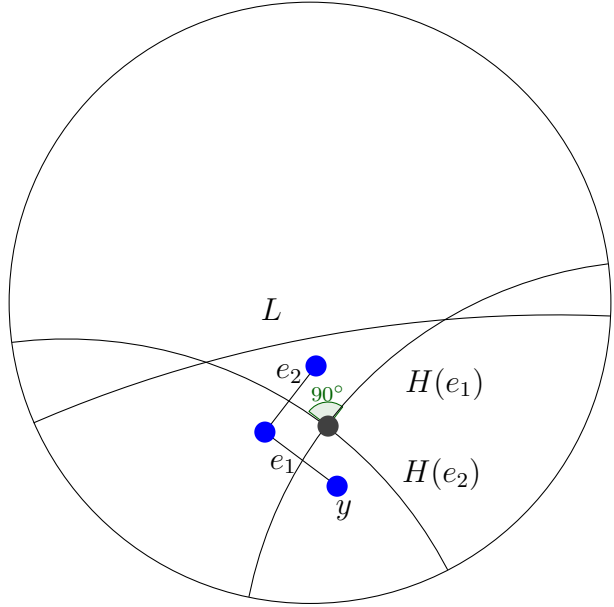
\begin{figure}
\centering
\definecolor{qqwuqq}{rgb}{0,0.39,0}
\definecolor{uququq}{rgb}{0.25,0.25,0.25}
\definecolor{qqqqff}{rgb}{0,0,1}
\begin{tikzpicture}[line cap=round,line join=round,>=triangle 45,x=1.0cm,y=1.0cm,scale=4]
\clip(-1.45,-1.32) rectangle (1.69,1.13);
\draw [shift={(0.06,-0.41)},color=qqwuqq,fill=qqwuqq,fill opacity=0.1] (0,0) -- (52.63:0.08) arc (52.63:142.63:0.08) -- cycle;
\draw(0,0) circle (1cm);
\draw [shift={(0.82,-4.43)}] plot[domain=1.53:1.978,variable=\t]({1*4.39*cos(\t r)+0*4.39*sin(\t r)},{0*4.39*cos(\t r)+1*4.39*sin(\t r)});
\draw [shift={(1.16,-1.25)}] plot[domain=1.694:2.944,variable=\t]({1*1.39*cos(\t r)+0*1.39*sin(\t r)},{0*1.39*cos(\t r)+1*1.39*sin(\t r)});
\draw [shift={(-0.81,-1.55)}] plot[domain=0.482:1.697,variable=\t]({1*1.43*cos(\t r)+0*1.43*sin(\t r)},{0*1.43*cos(\t r)+1*1.43*sin(\t r)});
\draw (0.09,-0.61)-- (-0.15,-0.43);
\draw (-0.15,-0.43)-- (0.02,-0.21);
\draw (0.05,-0.62) node[anchor=north west] {$y$};
\draw (-0.17,-0.47) node[anchor=north west] {$e_1$};
\draw (-0.15,-0.17) node[anchor=north west] {$e_2$};
\draw (0.28,-0.18) node[anchor=north west] {$H(e_1)$};
\draw (0.27,-0.48) node[anchor=north west] {$H(e_2)$};
\draw (-0.2, 0.05) node[anchor=north west] {$L$};
\begin{scriptsize}
\fill [color=qqqqff] (0.09,-0.61) circle (1pt);
\fill [color=qqqqff] (-0.15,-0.43) circle (1pt);
\fill [color=qqqqff] (0.02,-0.21) circle (1pt);
\fill [color=uququq] (0.06,-0.41) circle (1pt);
\draw[color=qqwuqq] (0.06,-0.3) node {$90\textrm{\degre}$};
\end{scriptsize}
\end{tikzpicture} 
\caption{Sketch of proof of Lemma \ref{fgtoqc}.}\label{correction}
\end{figure}

\begin{lemma} \cite[Correction to the proof of Theorem 3.1]{scott1985correction} \label{fgtoqc}
A closed, orientable, hyperbolic surface group $G$ is a finite index subgroup of $C(C_5)$, where $C_5$ is a cycle of length $5$. Moreover, for a suitable embedding $G \xhookrightarrow{} C(C_5)$, all finitely generated subgroups of $G$ are convex-cocompact in $C(C_5)$ with respect to the action on $DM(C_5)$.
\end{lemma}

\begin{remark}[Idea of proof]
Scott uses a different terminology, so it makes sense to summarise the proof. The natural generators of $C(C_5)$ act on the hyperbolic plane by reflections in  the sides of a right-angled pentagon. Translates of the pentagon give a tiling of the hyperbolic plane. Dual to this cell complex is a square complex $DM(C_5)$. Under this identification, the geodesic lines bounding the pentagons of the tiling become hyperplanes of $DM(C_5)$.
 
Suppose  $H$ is a finitely generated subgroup of the surface group $G = \pi_1 (\Sigma)$. Let $\Sigma_H$ be the covering space associated to $H$. By Lemma 1.5 in $\cite{scott1978subgroups}$,  there exists a closed, compact, incompressible subsurface $\Sigma' \subset \Sigma_H$ such that the induced map $\pi_1 \Sigma' \longrightarrow \pi_1\Sigma_H$ is surjective. Moreover, by \cite[Correction to the proof of Theorem 3.1]{scott1985correction} we can require $\Sigma'$ to have a geodesic boundary.

Let $\widetilde{\Sigma'}$ be the lift of  $\Sigma	'$ to $\mathbb{H}^2 = DM(C_5)$. Let $Y$ be the intersection of all half-spaces containing $\widetilde{\Sigma'}$. Suppose $y$ lies in $Y$, but not in $N_3(\widetilde{\Sigma'})$ and that $e_1, e_2$ are the first two edges of the combinatorial geodesic from  $y$ to $\widetilde{\Sigma'}$. Since $y \in Y$, both $H(e_1)$ and $H(e_2)$ intersect $\widetilde{\Sigma'}$. Consequently, $H(e_1)$ intersects $H(e_2)$ as $H(e_2)$ does not separate $H(e_1)$ from $\widetilde{\Sigma'}$. Call the intersection $y'$. The point $y$ is a centre of a pentagon and $y'$ is a vertex of the same pentagon, so the distance between them does not depend on $y$ (for example by specialness of $DM(\mathcal{G})$). See Figure \ref{correction}.

The closest boundary component $L$ of $\widetilde{\Sigma'}$ to $y$ is seen from $y'$ at more than the right angle (remember that the hyperplanes are geodesics). But such a point is within distance $\int_{t=0 }^{\pi/4} \frac{1}{\cos(t)}dt$ of $L$. To see this, take $L$ to be the vertical ray through $(0,0)$ in the upper half-plane model to the. Then the set of points with obtuse subtended angle is contained between rays $y = x$ and $y =-x$. Geodesic between these rays and $L$ is an arc of length $$\int_{t=0}^{\pi/4} \frac{r \sqrt{\cos^2(t) + \sin^2(t)}}{r \cos(t)} dt = \int_{t=0 }^{\pi/4} \frac{1}{\cos(t)}dt \,.$$

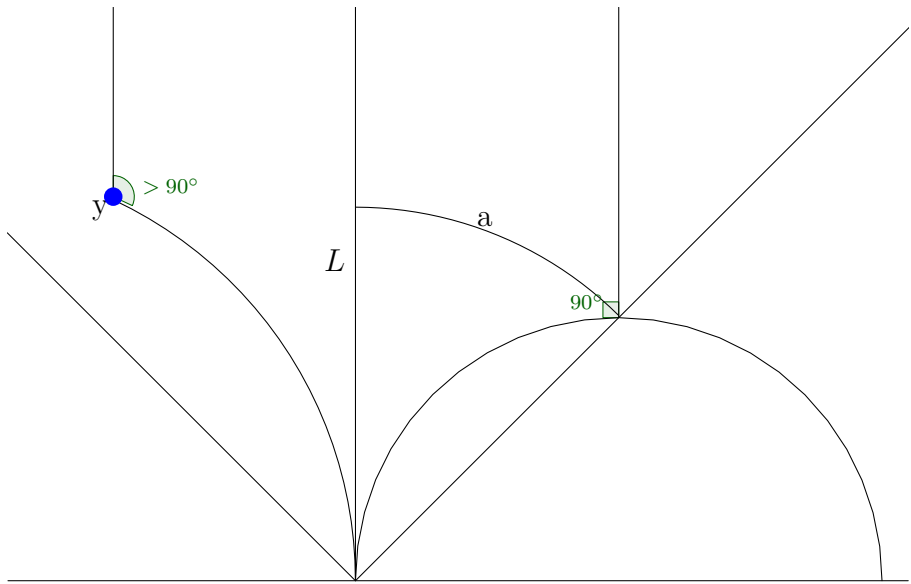
\begin{figure}
\centering
\definecolor{qqqqff}{rgb}{0,0,1}
\definecolor{qqwuqq}{rgb}{0,0.39,0}
\begin{tikzpicture}[line cap=round,line join=round,>=triangle 45,x=1.0cm,y=1.0cm,scale=7]
\clip(-0.66,-0.26) rectangle (1.05,1.09);
\draw[color=qqwuqq,fill=qqwuqq,fill opacity=0.1] (0.5,0.53) -- (0.47,0.53) -- (0.47,0.5) -- (0.5,0.5) -- cycle; 
\draw [shift={(-0.46,0.73)},color=qqwuqq,fill=qqwuqq,fill opacity=0.1] (0,0) -- (-25.22:0.04) arc (-25.22:90:0.04) -- cycle;
\draw [domain=-0.66:1.05] plot(\x,{(-0-0*\x)/1});
\draw (0,0) -- (0,1.09);
\draw [shift={(0.5,0)}] plot[domain=0:3.14,variable=\t]({1*0.5*cos(\t r)+0*0.5*sin(\t r)},{0*0.5*cos(\t r)+1*0.5*sin(\t r)});
\draw (0.5,0.5) -- (0.5,1.09);
\draw [domain=0.0:1.0541043647728972] plot(\x,{(-0--0.5*\x)/0.5});
\draw [domain=-0.6612882541543814:0.0] plot(\x,{(-0--0.5*\x)/-0.5});
\draw (-0.08,0.65) node[anchor=north west] {$L$};
\draw [shift={(0,0)}] plot[domain=0.79:1.57,variable=\t]({1*0.71*cos(\t r)+0*0.71*sin(\t r)},{0*0.71*cos(\t r)+1*0.71*sin(\t r)});
\draw [shift={(-0.8,0)}] plot[domain=0:1.13,variable=\t]({1*0.8*cos(\t r)+0*0.8*sin(\t r)},{0*0.8*cos(\t r)+1*0.8*sin(\t r)});
\draw (-0.46,0.73) -- (-0.46,1.09);
\draw (-0.54,0.76) node[anchor=north west] {$$ y' $$};
\draw (0.21,0.72) node[anchor=north west] {a};
\begin{scriptsize}
\draw[color=qqwuqq] (0.44,0.53) node {$90\textrm{\degre}$};
\fill [color=qqqqff] (-0.46,0.73) circle (0.5pt);
\draw[color=qqwuqq] (-0.35,0.75) node {$>90\textrm{\degre}$};
\end{scriptsize}
\end{tikzpicture}
\caption{If the angle subtended by $L$ from $y$ is obtuse, then $y$ is uniformly close to $L$.}
\end{figure}

Therefore $y'$ (and hence $y$) is at a uniformly bounded distance from $\widetilde{\Sigma'}$ and the action of $H$ on $Y$ is cocompact.
\end{remark}

\begin{corollary} \label{surfaces}
All finitely generated infinite index subgroups of closed, orientable, hyperbolic surface group $G$ are $\mathcal{A}$-separable in $G$.
\end{corollary}

\begin{proof}
By  Lemma \ref{fgtoqc}, finitely generated subgroups of $G$ are convex-cocompact in $C(C_5)$. By the Main Theorem \ref{main} they are $\mathcal{A}$-separable in $C(C_5)$. By Lemma \ref{fi}, they are $\mathcal{A}$-separable in $G$.
\end{proof}

%
%
%

\section{Proof of the Main Theorem}

\begin{definition} [Disjoint hyperplanes, bounding hyperplanes, positive half-space]
Let $X$ be a cube complex, $Y$ a convex subcomplex. Let $\mathcal{D}(Y)$ be the set of hyperplanes disjoint from $Y$. Let $\mathcal{B}(Y)$ be the set of hyperplanes bounding $Y$ (a hyperplane bounds $Y$ if it is dual to some $e$ with one endpoint in $Y$ and one not in $Y$).

If $H \in \mathcal{D}(Y)$, denote by \emph{$H^+$} the half-space of $X\backslash \backslash H$ containing $Y$.
\end{definition}

\begin{lemma} \label{bounding}
Any hyperplane $H$ bounding a convex subcomplex $Y$ in a $CAT(0)$ complex is disjoint from $Y$.
\end{lemma}
\begin{proof}
Let $e$ be an edge dual to $H$ such that $i(e) \in Y$ and $t(e) \notin Y$. Suppose $H$ intersects $Y$. Then there is an edge $e' \in Y$ dual to $H$. Without loss of generality $i(e')$ belongs to the same half-space of $H$ as $i(e)$. By convexity of $Y$ any (combinatorial) geodesic between $i(e)$ and $t(e')$ lies entirely in $Y$. A geodesic between two vertices in a $CAT(0)$ complex is precisely a path, which crosses each hyperplane separating the two vertices once \cite[Page 613]{sageev1995ends}. Therefore a concatenation of $e$ and a  geodesic from $t(e)$ to $t(e')$ is a geodesic from $i(e)$ to $t(e')$. The vertex $t(e)$ lies on this geodesic and hence belongs to $Y$ giving a contradiction.
\end{proof}
Recall that any intersection of half-spaces is convex and conversely any convex subcomplex is an intersection of the half-spaces containing it. Hence it is equivalent to specify a convex subcomplex or the half-spaces in which it is contained (or the set of disjoint hyperplanes if there can be no confusion about the choice of half-spaces, e.g. if only one choice gives a non-empty intersection).

\begin{definition}[Deletion, vertebra]
Suppose $G$ acts on a cube complex $X$ with core $Y$. Define \emph{deletion} as removing a bounding hyperplane $H_0$ and all its $G$-translates from $\mathcal{D}(Y)$.
The result of deletion of $H_0$ is $Y' = \cap_{H \in \mathcal{D}(Y) \setminus G. \{H_0 \} } H^+$.

The cube complex $V = H_0^- \cap Y'$ is called a \emph{vertebra}. See Figures \ref{fig:pentagons} and \ref{vertebra}.
\end{definition}

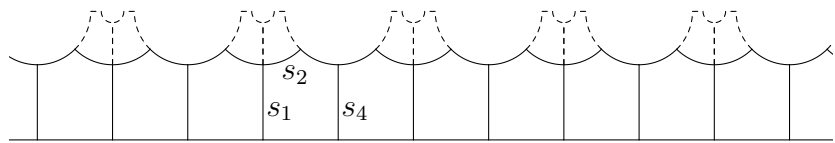
\begin{figure}
\centering
\begin{tikzpicture}[line cap=round,line join=round,>=triangle 45,x=1.0cm,y=1.0cm]
\clip(-4.37,-3.17) rectangle (6.57,5.42);
\draw (0,0)-- (0,1);
\draw (0,0)-- (1,0);
\draw (1,1)-- (1,0);
\draw [dash pattern=on 2pt off 2pt] (1,1)-- (1,1.56);
\draw [dash pattern=on 2pt off 2pt] (0.71,1.71)-- (0.85,1.71);
\draw [shift={(1,1.71)}] plot[domain=3.93:4.71,variable=\t]({1*0.71*cos(\t r)+0*0.71*sin(\t r)},{0*0.71*cos(\t r)+1*0.71*sin(\t r)});
\draw [shift={(0,1.71)}] plot[domain=4.71:5.5,variable=\t]({1*0.71*cos(\t r)+0*0.71*sin(\t r)},{0*0.71*cos(\t r)+1*0.71*sin(\t r)});
\draw [shift={(0,1.71)},dash pattern=on 2pt off 2pt]  plot[domain=-0.79:0,variable=\t]({1*0.71*cos(\t r)+0*0.71*sin(\t r)},{0*0.71*cos(\t r)+1*0.71*sin(\t r)});
\draw [shift={(1,1.71)},dash pattern=on 2pt off 2pt]  plot[domain=3.142:4.71,variable=\t]({1*0.15*cos(\t r)+0*0.15*sin(\t r)},{0*0.15*cos(\t r)+1*0.15*sin(\t r)});
\draw [shift={(1,1.71)}] plot[domain=3.93:4.71,variable=\t]({-1*0.71*cos(\t r)+0*0.71*sin(\t r)},{0*0.71*cos(\t r)+1*0.71*sin(\t r)});
\draw [shift={(2,1.71)}] plot[domain=4.71:5.5,variable=\t]({-1*0.71*cos(\t r)+0*0.71*sin(\t r)},{0*0.71*cos(\t r)+1*0.71*sin(\t r)});
\draw [shift={(2,1.71)},dash pattern=on 2pt off 2pt]  plot[domain=-0.79:0,variable=\t]({-1*0.71*cos(\t r)+0*0.71*sin(\t r)},{0*0.71*cos(\t r)+1*0.71*sin(\t r)});
\draw [shift={(1,1.71)},dash pattern=on 2pt off 2pt]  plot[domain=3.142:4.71,variable=\t]({-1*0.15*cos(\t r)+0*0.15*sin(\t r)},{0*0.15*cos(\t r)+1*0.15*sin(\t r)});
\draw (2,0)-- (2,1);
\draw (2,0)-- (1,0);
\draw [dash pattern=on 2pt off 2pt] (1,1)-- (1,1.56);
\draw [dash pattern=on 2pt off 2pt] (1.29,1.71)-- (1.15,1.71);
\draw [shift={(3,1.71)}] plot[domain=3.93:4.71,variable=\t]({-1*0.71*cos(\t r)+0*0.71*sin(\t r)},{0*0.71*cos(\t r)+1*0.71*sin(\t r)});
\draw [shift={(4,1.71)}] plot[domain=4.71:5.5,variable=\t]({-1*0.71*cos(\t r)+0*0.71*sin(\t r)},{0*0.71*cos(\t r)+1*0.71*sin(\t r)});
\draw [shift={(4,1.71)},dash pattern=on 2pt off 2pt]  plot[domain=-0.79:0,variable=\t]({-1*0.71*cos(\t r)+0*0.71*sin(\t r)},{0*0.71*cos(\t r)+1*0.71*sin(\t r)});
\draw [shift={(3,1.71)},dash pattern=on 2pt off 2pt]  plot[domain=3.142:4.71,variable=\t]({-1*0.15*cos(\t r)+0*0.15*sin(\t r)},{0*0.15*cos(\t r)+1*0.15*sin(\t r)});
\draw [shift={(3,1.71)}] plot[domain=3.93:4.71,variable=\t]({1*0.71*cos(\t r)+0*0.71*sin(\t r)},{0*0.71*cos(\t r)+1*0.71*sin(\t r)});
\draw [shift={(2,1.71)}] plot[domain=4.71:5.5,variable=\t]({1*0.71*cos(\t r)+0*0.71*sin(\t r)},{0*0.71*cos(\t r)+1*0.71*sin(\t r)});
\draw [shift={(2,1.71)},dash pattern=on 2pt off 2pt]  plot[domain=-0.79:0,variable=\t]({1*0.71*cos(\t r)+0*0.71*sin(\t r)},{0*0.71*cos(\t r)+1*0.71*sin(\t r)});
\draw [shift={(3,1.71)},dash pattern=on 2pt off 2pt]  plot[domain=3.142:4.71,variable=\t]({1*0.15*cos(\t r)+0*0.15*sin(\t r)},{0*0.15*cos(\t r)+1*0.15*sin(\t r)});
\draw (4,0)-- (4,1);
\draw (4,0)-- (3,0);
\draw (3,1)-- (3,0);
\draw [dash pattern=on 2pt off 2pt] (3,1)-- (3,1.56);
\draw [dash pattern=on 2pt off 2pt] (3.29,1.71)-- (3.15,1.71);
\draw (2,0)-- (3,0);
\draw [dash pattern=on 2pt off 2pt] (3,1)-- (3,1.56);
\draw [dash pattern=on 2pt off 2pt] (2.71,1.71)-- (2.85,1.71);
\draw [shift={(7,1.71)}] plot[domain=3.93:4.71,variable=\t]({-1*0.71*cos(\t r)+0*0.71*sin(\t r)},{0*0.71*cos(\t r)+1*0.71*sin(\t r)});
\draw [shift={(8,1.71)}] plot[domain=4.71:5.5,variable=\t]({-1*0.71*cos(\t r)+0*0.71*sin(\t r)},{0*0.71*cos(\t r)+1*0.71*sin(\t r)});
\draw [shift={(8,1.71)},dash pattern=on 2pt off 2pt]  plot[domain=-0.79:0,variable=\t]({-1*0.71*cos(\t r)+0*0.71*sin(\t r)},{0*0.71*cos(\t r)+1*0.71*sin(\t r)});
\draw [shift={(7,1.71)},dash pattern=on 2pt off 2pt]  plot[domain=3.142:4.71,variable=\t]({-1*0.15*cos(\t r)+0*0.15*sin(\t r)},{0*0.15*cos(\t r)+1*0.15*sin(\t r)});
\draw [shift={(7,1.71)}] plot[domain=3.93:4.71,variable=\t]({1*0.71*cos(\t r)+0*0.71*sin(\t r)},{0*0.71*cos(\t r)+1*0.71*sin(\t r)});
\draw [shift={(6,1.71)}] plot[domain=4.71:5.5,variable=\t]({1*0.71*cos(\t r)+0*0.71*sin(\t r)},{0*0.71*cos(\t r)+1*0.71*sin(\t r)});
\draw [shift={(6,1.71)},dash pattern=on 2pt off 2pt]  plot[domain=-0.79:0,variable=\t]({1*0.71*cos(\t r)+0*0.71*sin(\t r)},{0*0.71*cos(\t r)+1*0.71*sin(\t r)});
\draw [shift={(7,1.71)},dash pattern=on 2pt off 2pt]  plot[domain=3.142:4.71,variable=\t]({1*0.15*cos(\t r)+0*0.15*sin(\t r)},{0*0.15*cos(\t r)+1*0.15*sin(\t r)});
\draw [shift={(5,1.71)}] plot[domain=3.93:4.71,variable=\t]({1*0.71*cos(\t r)+0*0.71*sin(\t r)},{0*0.71*cos(\t r)+1*0.71*sin(\t r)});
\draw [shift={(4,1.71)}] plot[domain=4.71:5.5,variable=\t]({1*0.71*cos(\t r)+0*0.71*sin(\t r)},{0*0.71*cos(\t r)+1*0.71*sin(\t r)});
\draw [shift={(4,1.71)},dash pattern=on 2pt off 2pt]  plot[domain=-0.79:0,variable=\t]({1*0.71*cos(\t r)+0*0.71*sin(\t r)},{0*0.71*cos(\t r)+1*0.71*sin(\t r)});
\draw [shift={(5,1.71)},dash pattern=on 2pt off 2pt]  plot[domain=3.142:4.71,variable=\t]({1*0.15*cos(\t r)+0*0.15*sin(\t r)},{0*0.15*cos(\t r)+1*0.15*sin(\t r)});
\draw [shift={(5,1.71)}] plot[domain=3.93:4.71,variable=\t]({-1*0.71*cos(\t r)+0*0.71*sin(\t r)},{0*0.71*cos(\t r)+1*0.71*sin(\t r)});
\draw [shift={(6,1.71)}] plot[domain=4.71:5.5,variable=\t]({-1*0.71*cos(\t r)+0*0.71*sin(\t r)},{0*0.71*cos(\t r)+1*0.71*sin(\t r)});
\draw [shift={(6,1.71)},dash pattern=on 2pt off 2pt]  plot[domain=-0.79:0,variable=\t]({-1*0.71*cos(\t r)+0*0.71*sin(\t r)},{0*0.71*cos(\t r)+1*0.71*sin(\t r)});
\draw [shift={(5,1.71)},dash pattern=on 2pt off 2pt]  plot[domain=3.142:4.71,variable=\t]({-1*0.15*cos(\t r)+0*0.15*sin(\t r)},{0*0.15*cos(\t r)+1*0.15*sin(\t r)});
\draw (8,0)-- (8,1);
\draw (8,0)-- (7,0);
\draw (7,1)-- (7,0);
\draw [dash pattern=on 2pt off 2pt] (7,1)-- (7,1.56);
\draw [dash pattern=on 2pt off 2pt] (7.29,1.71)-- (7.15,1.71);
\draw (6,0)-- (6,1);
\draw (6,0)-- (7,0);
\draw [dash pattern=on 2pt off 2pt] (7,1)-- (7,1.56);
\draw [dash pattern=on 2pt off 2pt] (6.71,1.71)-- (6.85,1.71);
\draw (4,0)-- (5,0);
\draw (5,1)-- (5,0);
\draw [dash pattern=on 2pt off 2pt] (5,1)-- (5,1.56);
\draw [dash pattern=on 2pt off 2pt] (4.71,1.71)-- (4.85,1.71);
\draw (6,0)-- (5,0);
\draw [dash pattern=on 2pt off 2pt] (5,1)-- (5,1.56);
\draw [dash pattern=on 2pt off 2pt] (5.29,1.71)-- (5.15,1.71);
\draw [shift={(-1,1.71)}] plot[domain=3.93:4.71,variable=\t]({-1*0.71*cos(\t r)+0*0.71*sin(\t r)},{0*0.71*cos(\t r)+1*0.71*sin(\t r)});
\draw [shift={(0,1.71)}] plot[domain=4.71:5.5,variable=\t]({-1*0.71*cos(\t r)+0*0.71*sin(\t r)},{0*0.71*cos(\t r)+1*0.71*sin(\t r)});
\draw [shift={(0,1.71)},dash pattern=on 2pt off 2pt]  plot[domain=-0.79:0,variable=\t]({-1*0.71*cos(\t r)+0*0.71*sin(\t r)},{0*0.71*cos(\t r)+1*0.71*sin(\t r)});
\draw [shift={(-1,1.71)},dash pattern=on 2pt off 2pt]  plot[domain=3.142:4.71,variable=\t]({-1*0.15*cos(\t r)+0*0.15*sin(\t r)},{0*0.15*cos(\t r)+1*0.15*sin(\t r)});
\draw [shift={(-1,1.71)}] plot[domain=3.93:4.71,variable=\t]({1*0.71*cos(\t r)+0*0.71*sin(\t r)},{0*0.71*cos(\t r)+1*0.71*sin(\t r)});
\draw [shift={(-2,1.71)}] plot[domain=4.71:5.5,variable=\t]({1*0.71*cos(\t r)+0*0.71*sin(\t r)},{0*0.71*cos(\t r)+1*0.71*sin(\t r)});
\draw [shift={(-2,1.71)},dash pattern=on 2pt off 2pt]  plot[domain=-0.79:0,variable=\t]({1*0.71*cos(\t r)+0*0.71*sin(\t r)},{0*0.71*cos(\t r)+1*0.71*sin(\t r)});
\draw [shift={(-1,1.71)},dash pattern=on 2pt off 2pt]  plot[domain=3.142:4.71,variable=\t]({1*0.15*cos(\t r)+0*0.15*sin(\t r)},{0*0.15*cos(\t r)+1*0.15*sin(\t r)});
\draw [shift={(-3,1.71)}] plot[domain=3.93:4.71,variable=\t]({1*0.71*cos(\t r)+0*0.71*sin(\t r)},{0*0.71*cos(\t r)+1*0.71*sin(\t r)});
\draw [shift={(-4,1.71)}] plot[domain=4.71:5.5,variable=\t]({1*0.71*cos(\t r)+0*0.71*sin(\t r)},{0*0.71*cos(\t r)+1*0.71*sin(\t r)});
\draw [shift={(-4,1.71)},dash pattern=on 2pt off 2pt]  plot[domain=-0.79:0,variable=\t]({1*0.71*cos(\t r)+0*0.71*sin(\t r)},{0*0.71*cos(\t r)+1*0.71*sin(\t r)});
\draw [shift={(-3,1.71)},dash pattern=on 2pt off 2pt]  plot[domain=3.142:4.71,variable=\t]({1*0.15*cos(\t r)+0*0.15*sin(\t r)},{0*0.15*cos(\t r)+1*0.15*sin(\t r)});
\draw [shift={(-3,1.71)}] plot[domain=3.93:4.71,variable=\t]({-1*0.71*cos(\t r)+0*0.71*sin(\t r)},{0*0.71*cos(\t r)+1*0.71*sin(\t r)});
\draw [shift={(-2,1.71)}] plot[domain=4.71:5.5,variable=\t]({-1*0.71*cos(\t r)+0*0.71*sin(\t r)},{0*0.71*cos(\t r)+1*0.71*sin(\t r)});
\draw [shift={(-2,1.71)},dash pattern=on 2pt off 2pt]  plot[domain=-0.79:0,variable=\t]({-1*0.71*cos(\t r)+0*0.71*sin(\t r)},{0*0.71*cos(\t r)+1*0.71*sin(\t r)});
\draw [shift={(-3,1.71)},dash pattern=on 2pt off 2pt]  plot[domain=3.142:4.71,variable=\t]({-1*0.15*cos(\t r)+0*0.15*sin(\t r)},{0*0.15*cos(\t r)+1*0.15*sin(\t r)});
\draw [shift={(-7,1.71)}] plot[domain=3.93:4.71,variable=\t]({1*0.71*cos(\t r)+0*0.71*sin(\t r)},{0*0.71*cos(\t r)+1*0.71*sin(\t r)});
\draw [shift={(-8,1.71)}] plot[domain=4.71:5.5,variable=\t]({1*0.71*cos(\t r)+0*0.71*sin(\t r)},{0*0.71*cos(\t r)+1*0.71*sin(\t r)});
\draw [shift={(-8,1.71)},dash pattern=on 2pt off 2pt]  plot[domain=-0.79:0,variable=\t]({1*0.71*cos(\t r)+0*0.71*sin(\t r)},{0*0.71*cos(\t r)+1*0.71*sin(\t r)});
\draw [shift={(-7,1.71)},dash pattern=on 2pt off 2pt]  plot[domain=3.142:4.71,variable=\t]({1*0.15*cos(\t r)+0*0.15*sin(\t r)},{0*0.15*cos(\t r)+1*0.15*sin(\t r)});
\draw [shift={(-7,1.71)}] plot[domain=3.93:4.71,variable=\t]({-1*0.71*cos(\t r)+0*0.71*sin(\t r)},{0*0.71*cos(\t r)+1*0.71*sin(\t r)});
\draw [shift={(-6,1.71)}] plot[domain=4.71:5.5,variable=\t]({-1*0.71*cos(\t r)+0*0.71*sin(\t r)},{0*0.71*cos(\t r)+1*0.71*sin(\t r)});
\draw [shift={(-6,1.71)},dash pattern=on 2pt off 2pt]  plot[domain=-0.79:0,variable=\t]({-1*0.71*cos(\t r)+0*0.71*sin(\t r)},{0*0.71*cos(\t r)+1*0.71*sin(\t r)});
\draw [shift={(-7,1.71)},dash pattern=on 2pt off 2pt]  plot[domain=3.142:4.71,variable=\t]({-1*0.15*cos(\t r)+0*0.15*sin(\t r)},{0*0.15*cos(\t r)+1*0.15*sin(\t r)});
\draw [shift={(-5,1.71)}] plot[domain=3.93:4.71,variable=\t]({-1*0.71*cos(\t r)+0*0.71*sin(\t r)},{0*0.71*cos(\t r)+1*0.71*sin(\t r)});
\draw [shift={(-4,1.71)}] plot[domain=4.71:5.5,variable=\t]({-1*0.71*cos(\t r)+0*0.71*sin(\t r)},{0*0.71*cos(\t r)+1*0.71*sin(\t r)});
\draw [shift={(-4,1.71)},dash pattern=on 2pt off 2pt]  plot[domain=-0.79:0,variable=\t]({-1*0.71*cos(\t r)+0*0.71*sin(\t r)},{0*0.71*cos(\t r)+1*0.71*sin(\t r)});
\draw [shift={(-5,1.71)},dash pattern=on 2pt off 2pt]  plot[domain=3.142:4.71,variable=\t]({-1*0.15*cos(\t r)+0*0.15*sin(\t r)},{0*0.15*cos(\t r)+1*0.15*sin(\t r)});
\draw [shift={(-5,1.71)}] plot[domain=3.93:4.71,variable=\t]({1*0.71*cos(\t r)+0*0.71*sin(\t r)},{0*0.71*cos(\t r)+1*0.71*sin(\t r)});
\draw [shift={(-6,1.71)}] plot[domain=4.71:5.5,variable=\t]({1*0.71*cos(\t r)+0*0.71*sin(\t r)},{0*0.71*cos(\t r)+1*0.71*sin(\t r)});
\draw [shift={(-6,1.71)},dash pattern=on 2pt off 2pt]  plot[domain=-0.79:0,variable=\t]({1*0.71*cos(\t r)+0*0.71*sin(\t r)},{0*0.71*cos(\t r)+1*0.71*sin(\t r)});
\draw [shift={(-5,1.71)},dash pattern=on 2pt off 2pt]  plot[domain=3.142:4.71,variable=\t]({1*0.15*cos(\t r)+0*0.15*sin(\t r)},{0*0.15*cos(\t r)+1*0.15*sin(\t r)});
\draw (0,0)-- (-1,0);
\draw (-1,1)-- (-1,0);
\draw [dash pattern=on 2pt off 2pt] (-1,1)-- (-1,1.56);
\draw [dash pattern=on 2pt off 2pt] (-0.71,1.71)-- (-0.85,1.71);
\draw (-2,0)-- (-2,1);
\draw (-2,0)-- (-1,0);
\draw [dash pattern=on 2pt off 2pt] (-1,1)-- (-1,1.56);
\draw [dash pattern=on 2pt off 2pt] (-1.29,1.71)-- (-1.15,1.71);
\draw (-4,0)-- (-4,1);
\draw (-4,0)-- (-3,0);
\draw (-3,1)-- (-3,0);
\draw [dash pattern=on 2pt off 2pt] (-3,1)-- (-3,1.56);
\draw [dash pattern=on 2pt off 2pt] (-3.29,1.71)-- (-3.15,1.71);
\draw (-2,0)-- (-3,0);
\draw [dash pattern=on 2pt off 2pt] (-3,1)-- (-3,1.56);
\draw [dash pattern=on 2pt off 2pt] (-2.71,1.71)-- (-2.85,1.71);
\draw (-8,0)-- (-8,1);
\draw (-8,0)-- (-7,0);
\draw (-7,1)-- (-7,0);
\draw [dash pattern=on 2pt off 2pt] (-7,1)-- (-7,1.56);
\draw [dash pattern=on 2pt off 2pt] (-7.29,1.71)-- (-7.15,1.71);
\draw (-6,0)-- (-6,1);
\draw (-6,0)-- (-7,0);
\draw [dash pattern=on 2pt off 2pt] (-7,1)-- (-7,1.56);
\draw [dash pattern=on 2pt off 2pt] (-6.71,1.71)-- (-6.85,1.71);
\draw (-4,0)-- (-5,0);
\draw (-5,1)-- (-5,0);
\draw [dash pattern=on 2pt off 2pt] (-5,1)-- (-5,1.56);
\draw [dash pattern=on 2pt off 2pt] (-4.71,1.71)-- (-4.85,1.71);
\draw (-6,0)-- (-5,0);
\draw [dash pattern=on 2pt off 2pt] (-5,1)-- (-5,1.56);
\draw [dash pattern=on 2pt off 2pt] (-5.29,1.71)-- (-5.15,1.71);
\draw (-1.1,0.65) node[anchor=north west] {$ s_1 $};
\draw (-0.10,0.65) node[anchor=north west] {$ s_4 $};
\draw (-0.90,1.13) node[anchor=north west] {$ s_2 $};
\end{tikzpicture}

\caption{Cocompact subgroup $\langle s_1 s_4 \rangle < C(C_5)$ has a core dual to a row of pentagons. By deletion of the hyperplane labelled $s_2$, we get a larger core for the same group. }
\label{fig:pentagons}
\end{figure}

A vertebra is an intersection of two combinatorially geodesically convex sets, so it also is combinatorially geodesically convex. In particular, it is connected.

\begin{definition}[Acting without self-intersections]
We say $G$ acts \emph{without self-intersections} on a cube complex $X$, if $N(gH) \cap N(H) \neq \emptyset$ implies $gH = H$ for all hyperplanes $H$ of $X$ and $g \in G$.
\end{definition}

\begin{figure} 
\centering
\begin{tikzpicture} [scale=2] 
\draw [dash pattern=on 2pt off 2pt] (1,0) -- (0,0);
\draw (1,0) node[anchor=north] {$V$};
\draw (0,0) -- ($ (-1/2,{sqrt(3)/2}) $);
\draw (0,0) -- ($ (-1/2,{-sqrt(3)/2}) $);
\draw ($ (-1/2,{sqrt(3)/2}) $) -- ($ (-1/2,{sqrt(3)/2}) +(1/2,{sqrt(3)/2})$);
\draw ($ (-1/2,-{sqrt(3)/2}) $) -- ($ (-1/2,-{sqrt(3)/2}) +(1/2,{-sqrt(3)/2})$);
\draw ($ (-1/2,{sqrt(3)/2}) $) -- ($ (-1/2,{sqrt(3)/2}) -(1,0)$);
\draw ($ (-1/2,{-sqrt(3)/2}) $) -- ($ (-1/2,{-sqrt(3)/2}) -(1,0)$);
\draw (1/2,0.05) -- (1/2,-0.05) node[midway,anchor=north] {$H_0$};
\draw [thick,decoration={brace},decorate] (0,-2) -- (-3/2,-2) node[midway,anchor=north] {$Y'$};
\draw [thick,decoration={brace},decorate] (1,-2.3) -- (-3/2,-2.3) node[midway,anchor=north] {$Y$};
\end{tikzpicture}
\caption{Here $X$ is a $3$-regular tree and $G$ is trivial.}
\label{vertebra}
\end{figure}
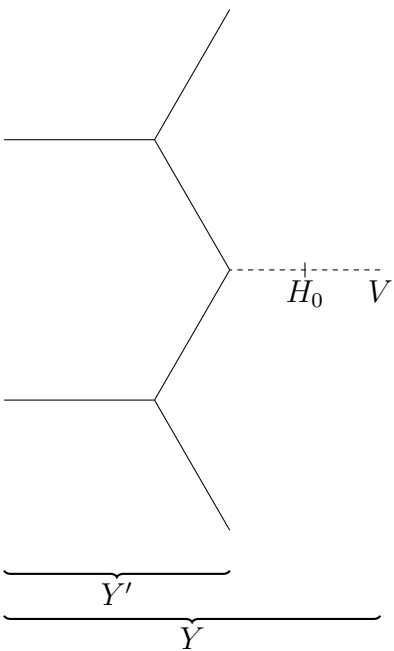

\begin{definition}[Special action]
An action of $G$ on a cube complex $X$ is \emph{special} if it is without self-intersections and moreover if $N(H) \cap N(K) \neq \emptyset$ and $H \cap gK \neq \emptyset$, then $H \cap K \neq \emptyset$.
\end{definition}

\begin{lemma} \label{lemmaA}
Suppose that $G$ acts without self-intersections on a locally compact $CAT(0)$ cube complex $X$ with core $Y$ and $H_0 \in \mathcal{B}(Y)$. Then the result $Y'$ of deletion of $H_0$ is also a core for $G$.
Let $G_{H_0} :=  \{g \in G | g.H_0 = H_0 \}$ be the stabiliser of $H_0$ in $G$.
If $C$ is a set of orbit representatives for the action of $G$ on the vertices of $Y$ and $D$ is a set of orbit representatives for the action of $G_{H_0}$ on the vertices of the vertebra $V = H_0^- \cap Y'$, then $C' = C \sqcup D$ is a set of orbit representatives for the action of $G$ on the vertices of $Y'$. Moreover, $Y' \subset N(Y)$.
\end{lemma}

\begin{proof}
Recall that $CAT(0)$ implies special.

First note that $\mathcal{D}(Y') = \mathcal{D}(Y) \setminus G. \{H_0 \}$ by definition and $\mathcal{B}(Y) \setminus G. \{H_0 \} \subset \mathcal{B}(Y')$ as a bounding hyperplane $Y$ still bounds $Y'$ unless it is a translate of $H_0$.

The set of half-spaces containing $Y$ is invariant under $G$, hence $Y'$ is invariant. The subcomplex $Y'$ is an intersection of half-spaces, hence convex.
Suppose $v \in Y' \setminus Y$. Let $v_0, v_1 \ldots v_k$ be a combinatorial geodesic from $v$ to $Y$ of shortest length with edges $e_1, \ldots e_k$ and suppose $k>1$. Let's $H_i$ be the hyperplane dual to $e_i$. Then as $v_{k-1} \notin Y$, we have $H_k \in G. \{H_0\}$.
Since $G$ acts on $X$ without self-intersections $H_{k-1} \notin G. \{H_0\}$. And $H_{k-1} \notin \mathcal{D}(Y')$, because $v_0, v_k \in Y'$ and $Y'$ is convex, so $e_{k-1} \in Y'$

Therefore $H_{k-1} \notin \mathcal{D}(Y)$. It must intersect $Y$, so it is not entirely contained in $H_k^-$ and it intersects $H_k$.
Because the cube complex is special, $H_k$ and $H_{k-1}$ do not interosculate. In particular, there is a square with two consecutive sides $e_{k-1}$ and $e_k$. Let $e'_j$ be the edge opposite $e_j$ in this square. By Lemma \ref{bounding} $H_{k-1}$ does not bound $Y$ and  $e_{k-1}' \in Y$. We can now construct a shorter path from $v_0$ to $Y$ with edges $e_1, \ldots, e_{k-2}, e'_{k}$. Contradiction.

So $k \leq 1$ and $Y'$ lies in a $1$-neighbourhood of $Y$ and therefore the action is cocompact.

There is a unique edge connecting $v \in Y' \setminus Y$ to $Y$ as any path of length $2$ is a geodesic or is contained in some square. In the first case by convexity of $Y$, we have $v \in Y$. In the second, $H_0 \notin \mathcal{D}(Y)$.

By invariance of $Y$, the $G$-translates of $V$ do not intersect $Y$. Suppose $v \in Y' \setminus Y$. There is a unique hyperplane in $G. \{H_0 \}$ dual to an edge $e_1$, which connects $v$ to $Y$, say $g.H_0$. Then $v$ belongs to a unique translate of $V$, namely $g.V$.
\end{proof}

\begin{corollary}
Let $\mathcal{G}$ be a finite simplicial graph.
If $K$ is a subgroup of a right-angled Coxeter group $C(\mathcal{G})$ and it acts on the Davis-Moussong complex with core $Y$, then deletion produces another core.
\end{corollary}

\begin{proof}
The Davis-Moussong complex $DM(\mathcal{G})$ is a $CAT(0)$ cube complex, hence simply connected and special. The action of $C(\mathcal{G})$ on it preserves labels. In this complex any two consecutive edges have distinct labels, so the action is without self-intersections. The restriction to $K$ is also without self-intersections.
\end{proof}

\begin{lemma} \label{intersection}
Suppose $G$ acts on a $CAT(0)$ cube complex $X$ with core $Y$.
If $Y' \subset X$ is constructed from $Y$ using a deletion  of $H = H(e)$, then each edge  in $V = H^- \cap Y'$ is dual to a hyperplane intersecting $H$.
\end{lemma}
\begin{proof}
Let $e'$ be an edge in $V$ and $H'$ a hyperplane dual to $e'$. If $H' \cap H = \emptyset$, $H'$ is contained entirely in $H^-$. But then $H'$ is disjoint from $Y$. In particular one of the endpoints of $e'$ is in the opposite half-space of $X \backslash \backslash H'$ than $Y$.

Since $Y'$ is the intersection of all half-spaces containing $Y$ with the exception of the  $G$-translates of $H^+$, the hyperplane $H'$ is $gH$ for some $g \in G$.

The subcomplex $Y$ is $G$-invariant and $H$ bounds $Y$, hence $H'$ bounds $Y$. This contradicts $H' \subset H^-$.
\end{proof}

\begin{corollary} \label{commutation}
Suppose $G<C(\mathcal{G})$ acts on $DM(\mathcal{G})$ with core $Y$.
If $Y' \subset X$ is constructed from $Y$ using a deletion of $H = H(e)$, then each edge in $V = H^- \cap Y'$ has a label which commutes with the label of $e$.
\end{corollary}

\begin{definition}[Deletion along a path, deletion with labels, tail]
Suppose $Y$ is a subcomplex of a $CAT(0)$ complex $X$ and a core for the action of $G$ on $X$. Suppose $p=e_1 e_2 \ldots e_n$ is a path in $X$, which starts in $Y$. \emph{The deletion of hyperplanes along the path $p$} is a subcomplex $Y' = \cap H^+$, where $H$ goes over hyperplanes disjoint from $Y$ and from $G.p$.

Suppose additionally that edges of $X$ are labelled in such a way that for every vertex and every label, there is precisely one edge starting at that vertex of the given label.  Suppose $v \in Y$, and $s_1, s_2, \ldots s_n$ is a sequence of edge labels, then \emph{the deletion with labels $s_1, s_2, \ldots s_n$ at $v$} is the \emph{deletion of hyperplanes along $p$}, where $p$ is a path $e_1, e_2, \ldots e_n$ starting at $v$ with $e_i$ labelled $s_i$.

Suppose $Y_n$ was built from $Y_0$ using a series of deletion of hyperplanes $H_1, \ldots H_n$. We call $T=Y_n \cap H_1^-$ \emph{a tail}.
\end{definition}

\begin{lemma} \label{reduce}
Suppose $\mathcal{G}$ is a finite simplicial graph. Suppose $\mathcal{G}^c$ is connected, $| \mathcal{G}| >1$ and $H$ acts on $DM(\mathcal{G})$ with a core $Y \subsetneq DM(\mathcal{G})$. Then there exists a core $Y'$ which can be obtained from $Y$ by deletion along a path $e_1, e_2 \ldots e_n$ with  the vertebra $Y' \cap H(e_n)^-$ a single vertex.
\end{lemma}

\begin{remark}
The hypothesis that $\mathcal{G}^c$ is connected is necessary. Consider the situation when $\mathcal{G}$ is a square.
Then $C(\mathcal{G}) = D_\infty \times D_\infty$ and $DM(\mathcal{G})$ is the standard tiling of $\mathbb{R}^2$. Let $H = D_\infty$ be the subgroup generated by two non-commuting generators of $C(\mathcal{G})$. The invariance of the core and cocompactness of the action imply that any core for $H$ is of the  form $\mathbb{R} \times [k,l]$ for some $k,l \in \mathbb{Z}$.

Every hyperplane intersecting such a core divides it into two infinite parts.
\end{remark}

\begin{proof}[Proof of Lemma \ref{reduce}]
Since $Y$ is a proper subcomplex, there exists $e_1$ be such that $H(e_1) = H_1$ bounds $Y$. Let $v_0$ be the endpoint of $e_1$, which lies in $Y$. Let $v_1$ be the other endpoint. Say the label of $e_1$ is $s_1$. Let $Y_1$ be a cube complex obtained from $Y$ by deletion of $H_1$.

Let $S_1$ be the set of generators labelling the edges of vertebra $V_1$. Then by Corollary \ref{commutation}, $s_1$ commutes with all generators in $S_1$.

 If $e_2 \notin V_1$ is an edge with endpoint $v_1$, whose label $s_2$ does not commute with $s_1$, we can define $H_2, Y_2, V_2$ and $S_2$ similarly as before.  Just as before the generators of $S_2$ commute with $s_2$.

The hyperplanes $H_1$ and $H_2$ do not intersect, so $N(H_2) \subset H_1^-$.  There is an inclusion of $V_2$ into $V_1$ given by sending a vertex of $V_2$ to the unique vertex of $V_1$ to which it is connected by an edge labelled $s_2$. Extending this map to edges and cubes is a label preserving map between cube complexes $V_2$ and $V_1$.
It follows that $S_2$ is a (not necessarily proper) subset of $S_1$.

We will now show that, by a series of such operations, we can reach a situation where $S_n = \emptyset$. I.e. the vertebra $V_n$ is a single vertex.

Suppose we have already applied deletion $i$ times and $S_i$ is non-empty. We will use a series of deletions to get $S_{k+1} \subsetneq S_{k} \subset S_{k-1} \subset \ldots \subset S_{i+1} \subset S_i$. By an abuse of notation, we'll identify the vertices of $\mathcal{G}^c$ with the labels and with the generators of the right-angled Coxeter group. (Rather than having a generator $s_v$ for every vertex $v \in V(\mathcal{G})$ and using these as labels.)

Since the group does not split as a product, there exists some $a \in S_i$ and $b \notin S_i$ which do not commute.
Since $\mathcal{G}^c$ is connected, there exists a vertex path $s_{i-1}, \ldots s_k = b$ in $\mathcal{G}^c$ from the vertex $s_{i-1}$, which is the label of the hyperplane we removed last.

Successive generators in this path do not commute. Indeed assume that $s_j$ and $s_{j+1}$ commute. Take $v \in DM(\mathcal{G})$, let $e_1, e_2, e_3, e_4$ be edges of the path starting at $v$ with labels $s_j, s_{j+1}, s_j, s_{j+1}$. This is a closed loop, since $s_j$ and $s_{j+1}$ commute. The hyperplane $H(e_1)$ separates $v$ from $s_j s_{j+1} v$, so it has to be dual to one of $e_3$ and $e_4$. Parallel edges have the same labels so $H(e_1) = H(e_3)$. Similarly $H(e_2) = H(e_4)$. The hyperplane $H(e_1)$ separates $e_2$ from $e_4$, so $H(e_1)$ and $H(e_2)$ have to cross. The Davis-Moussong complex is special, so there is a square where $e_1$ and $e_2$ are successive edges. By construction of the complex, $s_j$ and $s_{j+1}$ are connected by an edge is $\mathcal{G}$. This contradicts adjacency of $s_j$ and $s_{j+1}$ in $\mathcal{G}^c$.

Apply deletion of hyperplanes labelled $s_i, \ldots , s_k$ starting at some vertex of $v \in V_{i-1}$. Note that the $j$th hyperplane we remove belongs in a subset of $\mathcal{B}(Y_{j-1})$ as $s_i \ldots s_{j-1} v \in V_{j-1}$ and $s_j$ does not commute with $s_{j-1}$. Moreover, $S_j = \{s \in S_{j-1} : ss_j = s_j s\}$. In particular, $S_{k+1} \subset S_i$ and $a$ does not belong to $S_{k+1}$ as $as_k \neq s_ka$. Similarly, the hyperplane $H'$ dual to edge between $v$ and $s_{j+1}$ is dual.

Therefore $S_{k+1}$ is a proper subset of $S_i$ and we can continue this process until we get an empty $S_n$.
\end{proof}

\begin{remark} \label{controllabel}
We can even control the label of the hyperplane which was removed last. Indeed, if the last removed hyperplane had label $s_i$, and $b$ is some other generator, pick a vertex path between $s_i$ and $b$ in $\mathcal{G}^c$. Then remove hyperplanes labelled by vertices on this path, starting at the unique vertex of a vertebra.
\end{remark}

By Lemma \ref{lemmaA} there is a set of orbit representatives $K$ for the action of $G$ on $Y_n$ with $T \subset K$.

Haglund shows the following \cite[Proof of Theorem A]{haglund2008finite}.

\begin{lemma} \label{Scott}
Suppose $G < C(\mathcal{G})$ acts on $DM(\mathcal{G})$ with a core $Y$ and with a set of orbit representatives $K$. Let $\Gamma_0 <C(\mathcal{G})$ be generated by the reflections in the hyperplanes bounding $Y$. Let $\Gamma_1 = \Gamma_1 (Y)= \langle G, \Gamma_0 \rangle$.
Then $Y$ is a fundamental domain for the action of $\Gamma_0$ on $X$ and $K$ is a set of orbit representatives for the action of $\Gamma_1$ on $X$.
\end{lemma}

Let $C(\mathcal{G})$ act on the right cosets of $\Gamma_1 < C(\mathcal{G})$. We have that $s \in S$ sends $\Gamma_1 g$ to $\Gamma_1 g s = (\Gamma_1 g s g^{-1}) g$. But $gs g^{-1}$ is a reflection in the hyperplane $H(ge_s)$. By definition of $\Gamma_0$ if $H(ge_s)$ bounds $Y$, $gs g^{-1} \in \Gamma_0$ and $\Gamma_1 g$ is fixed by $s$.

Moreover, if $K = \{g_1 v_0, \ldots ,g_n v_0 \}$, then $\{g_1, \ldots g_n\}$ is a set of right coset representatives for $\Gamma_1$.

We will first prove that by a suitable sequence of deletions, we can satisfy the conditions of Jordan's theorem. It follows that we can construct quotients that are either alternating or symmetric.

\begin{definition}
If $Y$ is a subset of a cube complex $X$, then $N_1(Y)$ is a union of closed cubes, which have non-empty intersection with $Y$. We define inductively $N_i(Y) = N_1(N_{r-1}(Y))$.
\end{definition}

If $Y$ is convex, then so is $N_r(Y)$ (as a neighbourhood is obtained by removing bounding hyperplanes and therefore it is an intersection of convex subcomplexes). And if $H$ acts cocompactly on $Y$, it still acts cocompactly on $N_r(Y)$ assuming that $X$ is locally compact.

\begin{proposition}
Let $C(\mathcal{G})$ be the right-angled Coxeter group associated to $\mathcal{G}$ a finite simplicial graph, $|\mathcal{G}| >2$ , and suppose that $H < C(\mathcal{G})$ acts on the associated Davis-Moussong complex with a proper core $Y$. Let $\mathcal{C}$ be the class of symmetric and alternating groups. If $\mathcal{G}^c$ is connected, then $H$ is $\mathcal{C}$-separable.
\end{proposition}

\begin{proof}
As $H$ acts with a proper core, there exists a generator of $C(\mathcal{G})$ not contained in $H$. Say $s_0 \notin H$.

Suppose $\gamma_1, \ldots, \gamma_n \notin H$.

Fix $v \in Y$. Without loss of generality, we may assume that $Y$ contains $N(v)$ and  $\gamma_i v$ for all $i$ (otherwise replace $Y$ with $N_r(Y)$ for a sufficiently large $r$). Moreover, by Lemma \ref{reduce} we may assume that there exists a hyperplane $H_0 \notin \mathcal{D}(Y)$ with $|H_0^-  \cap Y| = 1$ and by the Remark \ref{controllabel} we may assume that the label of $H_0$ is $s_0$.

As $\mathcal{G}^c$ is connected, there exists a generator $s_1$ not commuting with $s_0$. Let $v_0$ be the unique vertex of $H_0^- \cap Y$. Let $e_1$ be the edge starting at $v_0$ with a label $s_1$. Obtain $Y_1$ by deleting $H(e_1)$ from the boundary of $Y$. By Lemma \ref{lemmaA} $Y' \subset N(Y)$. If $v_1 \in Y' \cap H(e_1)^-$, then there is an edge starting at $v$ with the other endpoint in $Y$. This edge is labelled $s_1$ and is dual to $H(e_1)$. Now $N(H(e_1)) \cap Y = v_0$ since otherwise $H(e_1)$ would have to intersect $H_0$ and $s_1$ would commute with $s_0$. Therefore $v_1$ is uniquely determined as the other endpoint of $e_1$.

Continue this by taking $e_i$ to be the edge starting at $v_{i-1}$ labelled $s_0$ for even $i$ and $s_1$ for odd $i$ and let $v_i$ be the other endpoint of $e_i$. Let $Y_i$ be $Y_{i-1}$ with $H(e_i)$ deleted from the boundary. Let $Y' = Y_k$ with $k$ to be specified later. 

Let $\Gamma_0$ be the group generated by the reflections in the hyperplanes bounding $Y'$. Let $\Gamma_1 = \langle \Gamma_0, H \rangle$.
Then $[C(\mathcal{G}): \Gamma_1] = |H \setminus Y'|$, where $|H \setminus Y'|$ denotes the number of vertices of $H \setminus Y'$. Every successive vertebra consists of a single vertex, so by Lemma \ref{lemmaA} $|H \setminus Y_{i+1}| = |H \setminus Y_i| + 1$. We can choose $k$ to make $|H \setminus Y'|$ a prime.
As $V(\Gamma_1 \setminus C(\mathcal{G}))$ is in a natural bijection with $V(H \setminus Y')$ and $\gamma_i v \notin H.v$, we may choose $\gamma_i$ as one of the coset representatives. $\Gamma_1 . \gamma_i \neq \Gamma_1$. Since $H < \Gamma_1$, for every $h \in H$ we have $\Gamma_1 h = \Gamma_1$. Therefore $\gamma_i$ does not act as an element of $H$, i.e. $f(\gamma_i) \notin f(H)$.

Let $s_3$ be a generator distinct from $s_0$ and $s_1$.
By the remark after Lemma \ref{Scott}, we can identify the right cosets of $\Gamma_1$ with orbits of $Y'$ under the action of $H$ and we can read of the action from the geometry as follows. Pick $v \in Y$ in an orbit corresponding to $\Gamma_1 g$, let $u$ be a vertex connected to $v$ by an edge labelled $s$. If $u \notin Y'$, then $\Gamma_1 g = \Gamma_1 g s$. If $u \notin Y'$, then $\Gamma_1 g s$ is the coset corresponding to $H.u$.
Since the tail contains no edge labelled $s_2$, every coset corresponding to a vertex in the tail is fixed by $s_2$.

 So $s_2$ moves at most $|H \setminus Y|$ elements. By taking $k$ large enough while $|H \setminus Y'|$ is still a prime, we may ensure that the conditions of Jordan's lemma are satisfied (the primitivity follows from transitivity and a non-existence of non-trivial partition of a prime number of elements into sets of the same size).
\end{proof}

\section{Changing parity}
We shall now prove that we may force the action to be alternating (similarly we can force it to be symmetric). Let $\mathcal{G}$ be a non-discrete locally compact finite graph throughout this section.

\begin{definition}
Suppose $Y$ is a core for an action of $G< C(\mathcal{G})$ on a $DM(\mathcal{G})$.
\emph{The parity of $s_i$ with respect to the core $Y$} is the parity of $s_i$ acting on the right cosets of $\Gamma_1 (Y)$, where $\Gamma_1(Y)$ is the finite index subgroup of $C(\mathcal{G})$ generated by $G$ and the reflections in the hyperplanes bounding $Y$.
\end{definition}

We will modify the construction of the tail in order to make each $s_i$ act as an even permutation (or we will make at least one of $s_i$ acts as an odd permutation).

Suppose $g.v_0$ is in the tail. If the edge between $g.v_0$ and $gs.v_0$ is in the tail, then $g.v_0$ and $gs.v_0$ map to distinct vertices in $\Gamma_1 \setminus X$, hence $\Gamma_1 g \neq \Gamma_1 g s$.

If $gs.v_0$ is not in the tail, then the hyperplane dual to this edge bounds $Y$ and the reflection in this hyperplane belongs to $\Gamma_1$. Therefore  $\Gamma_1  = \Gamma_1 g s g^{-1}$ or equivalently $\Gamma_1 g = \Gamma_1 g s$.

More precisely, suppose $H$ acts with core $Y$ and $Y'$ is the core resulting from deletion of $H_0, \ldots, H_k$, and the label of $H_i$ is $s_i$. Moreover assume $H_0 \cap Y'$ is a single edge.

Then the parity of $s_1$ with respect to $Y'$ is the sum of the parity of $s_1$ with respect to $Y$ and the number of edges labelled $s_1$ in $H_0^- \cap Y'$.
So we can control the parity of $s_1$ by changing the number of edges with label $s_1$ in the tail.
Suppose that the conditions of Jordan's theorem are satisfied with a margin $M$ (i.e. the conditions are satisfied even if $s_3$ moves $|H \setminus Y| + M$ elements). Taking $M = (|\mathcal{G}|-2) (2d+1)+16$, where $d$ is the diameter of $\mathcal{G}^c$ will be sufficient.

First let us show that we can deal with parity of all generators other than $s_1$ and $s_2$.

\begin{lemma}\label{parity}
For any $i \in I \setminus \{1,2\}$, if the tail of $Y$ is a path with labels $s_1, s_2, \ldots s_1, s_2, s_1$ of length at least $2d_{\mathcal{G}^c}(v_1,v_i)+1$ starting at vertex $V$, then there exists a core $Y'$ such that in the associated action the parity of $s_i$ changed and the parities of no $s_j$ changed for $j \in I \setminus \{1,2,i\}$.
Moreover, $|H \setminus Y| = |H \setminus Y'|$ and $Y'$ contains a tail of the same length as $Y$ and the labels of these two paths are the same with the exception of a subpath labelled $s_1, s_2, \ldots s_1, s_2, s_1$ of length $2d_{\mathcal{G}^c}(v_1,v_i)+1$.
\end{lemma}

\begin{figure}
\centering
\begin{tikzpicture}[scale=0.8]
\draw (0,0) -- (10,0) ;
\draw (2,4) -- (-6,4);
\foreach \x in {0,4,8} \draw (\x,0) -- node[anchor=east] {$s_1$} (\x,1);
\foreach \x in {2,6,10} \draw (\x,0) -- node[anchor=east] {$s_2$} (\x, 2);
\foreach \x in {2,-2,-6} \draw (\x,4) -- node[anchor=east] {$s_2$} (\x,2);
\foreach \x in {0,-4} \draw (\x,4) -- node[anchor=east] {$s_1$} (\x, 3);

\foreach \x in {-2,2,6} \draw ({\x-1},2) -- ({\x+1},2);
\foreach \x in {1,5,9} \draw  (\x,0) node[anchor=north] {$s_3$};
\foreach \x in {3,7} \draw  (\x,0) node[anchor=north] {$s_4$};
\foreach \x in {-5.5,-2.5,-1.5} \draw  (\x,2) node[anchor=north] {$s_5$};
\foreach \x in {1.5,2.5,5.5,6.5,9.5} \draw  (\x,2) node[anchor=south] {$s_5$};
\foreach \x in {-3,1} \draw  (\x,4) node[anchor=south] {$s_3$};

\draw (0,2) node[anchor=south west] {$s_4$};
\draw (0,2) node[anchor=south east] {$s_4$};
\draw (0,2) node[anchor=north west] {$s_4$};
\draw (-4,2) node[anchor=south west] {$s_4$};
\draw (-4,2) node[anchor=south east] {$s_4$};
\draw (4,2) node[anchor=north west] {$s_4$};
\draw (4,2) node[anchor=north east] {$s_4$};
\draw (8,2) node[anchor=north west] {$s_4$};
\draw (8,2) node[anchor=north east] {$s_4$};

\draw (-6,2) -- (-5,2);
\draw (10,2) -- (9,2);
\draw (-3,2) arc (0:180:1);
\draw (0,1) arc (-90:180:1);
\draw (3,2) arc (-180:0:1);
\draw (7,2) arc (-180:0:1);

\end{tikzpicture}
\caption{Sketch of the situation in Lemma \ref{parity}, where $\Gamma$ is a cycle of length $5$ and $i=5$. Here we've drawn the hyperplanes. The cube complex would be the dual picture. The lower five squares are the old tail and the upper four squares form the end of the new tail.} 
\end{figure}

\begin{proof}
Say $v_1 = v_{i_0}, v_{i_1}, \ldots v_{i_d} = v_i$ is a path in $\mathcal{G}^c$ of the shortest length. Let $Y'$ be a subcomplex built using deletions of hyperplanes $s_{i_0}, s_{i_1}, \ldots s_{i_d}, s_{i_{d-1}}, \ldots, s_{i_0}, s_2, s_1, \ldots s_1$ starting at $v$.

Compared to $Y$, the tail of this complex contains two more edges labelled $s_{j_i}$ for $0 < j < d$. It also contains an extra edge labelled $s_{i_d} = s_i$, so the parity of $s_i$ changed and the parity of other generators $s_j $ remains the same for $j \neq 1,2,i$.
\end{proof}

Now let's change the parity of a generator that appears in the tail.

\begin{lemma} \label{squares}
If the tail of $Y$ contains a path with labels $s_1, s_2, \ldots s_1, s_2, s_1$ of length  at least $7$, then there exists a core $Y'$ such that in the associated action only the parity of $s_1$ changed.
Moreover, $|H \setminus Y| = |H \setminus Y'|$ and $Y'$ is built from the same complex as $Y$ using a sequence of deletions, whose labels agree with that of $Y$ with the exception of $5$ deletions. (We allow a deletion to be replaced by no deletion.)

\end{lemma}
\begin{figure}
\centering
\begin{tikzpicture}
\filldraw (0,0) node[anchor=north] {$v_1$} circle (0.05);
\filldraw (1,0) node[anchor=north] {$v_2$} circle (0.05);
\filldraw (1/2,{sqrt(3)/2}) node[anchor=south] {$v_3$} circle (0.05);
\draw (1,0) -- (1/2,{sqrt(3)/2});

\draw (2,0) -- node[anchor=south east] {$s_2$} ({2+sqrt(2)},{sqrt(2)}) -- node[anchor= south west] {$s_1$} ({2+2*sqrt(2)},0) -- node[anchor=south east] {$s_2$} ({2+3*sqrt(2)},{sqrt(2)});

\draw (7,0) -- node[anchor=south east] {$s_2$} ({7+sqrt(2)},{sqrt(2)}) -- node[anchor=south west] {$s_3$} ({7+2*sqrt(2)},0) --  node[anchor=north west] {$s_2$} ({7+sqrt(2)},{-sqrt(2)}) -- node[anchor=north east] {$s_3$} (7,0);

\end{tikzpicture}
\caption{A sketch of the subgraph of $\mathcal{G}$ spanned by $v_1,v_2$ and $v_3$, the segment of the old tail and the new square which replaces this segment in the case 1 of the proof of Lemma \ref{squares}. } 
\end{figure}
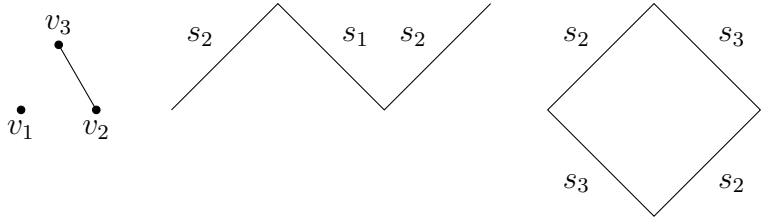

\begin{proof}

\begin{enumerate}

\item Suppose there exists distinct $s_3$ and $s_4$ which commute mutually but neither of which commutes with $s_1$.
Then instead of the deletion of the hyperplanes labelled $s_2, s_1, s_2$, delete the hyperplanes labelled $s_3, s_4$. This creates a square. Continue building the tail starting from one of the vertices of the square using the deletions of the hyperplanes with the same labels as before. The new tail contains  two fewer $s_2$ labels, two more of $s_3$ and two more of $s_4$ and one fewer $s_1$ (or the same number of $s_2$ and two more $s_3$, if $s_2 = s_4$ etc.). Hence only the parity of $s_1$ changed.

To be precise, we need to take the path labelled $s_2, s_1, s_2$ which is a subpath of a path labelled $s_1, s_2, s_1, s_2, s_1$ in the tail, as otherwise deleting a hyperplane labelled $s_3$ could introduce more than just a side of a square. Similarly for the other cases in this proof.

\item Suppose there is some $s_3$ commuting with $s_1$, but not $s_2$.
Then instead of the deletion of the hyperplanes labelled $s_1, s_2, s_1, s_2, s_1$, delete the hyperplanes labelled $s_1, s_3$ and then delete the hyperplanes labelled $s_2$ at two of the vertices of the square. This creates a square with two spurs.
Continue building the tail starting from the remaining vertex of the square. The new tail contains the same number of $s_2$ labels, two more of $s_3$ and one fewer $s_1$. Hence only the parity of $s_1$ changed.

\item Lastly, if neither of the above cases holds, then $\mathcal{G}$ consists of $v_1$, isolated vertices $I$, vertices $S_1$ at distance $1$ from $v_1$ and vertices $S_2$ at distance $2$ from $v_1$. Moreover, there exists a vertex adjacent to $v_1$ as the graph is non-discrete. Every vertex adjacent to $v_1$  is adjacent to  $v_2$, so $v_2 \in S_2$.

The induced graph on vertices of $S_2$ is discrete because every edge intersects $S_1$. Take any $u \in S_1$. Consider a path from $u$ to $v_1$ in $\mathcal{G}^c$. Somewhere along this path we go from a vertex, which is connected to both $v_1$ and $v_2$ to a vertex which is connected to neither.
Therefore there are $s_3$ and $s_4$ such that $s_3$ commutes with $s_1$ and $s_2$ and $s_4$ does not commute with any of $s_1, s_2$ and $s_3$. Now instead of the deletion of the hyperplanes labelled $s_1, s_2, s_1, s_2, s_1$ delete the hyperplanes labelled $s_4, s_1, s_3, s_4$. This creates a square with labels $s_1, s_3, s_1, s_3$. Continue building the tail. We have one fewer $s_1$, two fewer $s_2$ and two more of each $s_3$ and $s_4$.

\end{enumerate}
Let $Y'$ be the new subcomplex. By construction $|H \setminus Y| = |H \setminus Y'|$ and the sequences of labels of deleted hyperplanes for the two complexes differ at no more than $5$ places.
\end{proof}

Using Lemmas \ref{parity} and \ref{squares}, we can now modify segments of the tail to make the parity of all elements even. This completes the proof of the main theorem.

\bibliographystyle{alpha}
\bibliography{mybib}

\end{document}